\documentclass{amsart}
\usepackage{amsmath,amsfonts,euscript,amscd,amsthm,amssymb,graphics,upref}
\usepackage{mathrsfs}
\newcommand{\Pscr}{\mathscr{P}}
\usepackage{mathabx}
\newcommand{\flouc}[1]{\overset{{#1}}{\righttoleftarrow}}
\usepackage[all]{xy}

\usepackage{latexsym}
\usepackage{amsfonts,amssymb} 
\usepackage{graphicx}
\usepackage{tikz}
\usepackage{caption}
\usepackage{subcaption}
\usepackage{enumitem}

\newcommand{\field}[1]{\mathbb{#1}}
\newcommand{\A}{\field{A}}
\newcommand{\C}{\field{C}}

\newcommand{\N}{\field{N}}

\newcommand{\Q}{\field{Q}}

\newcommand{\Z}{\field{Z}}

\newcommand{\pgoth}{{\ensuremath{\mathfrak{p}}}}
\newcommand{\qgoth}{{\ensuremath{\mathfrak{q}}}}

\newcommand{\bP}{\mathbf{P}}
\newcommand{\Div}{	\operatorname{\text{\rm Div}}}
\renewcommand{\div}{	\operatorname{\text{\rm div}}}
\newcommand{\Spec}{	\operatorname{\text{\rm Spec}}}
\newcommand{\trdeg}{\operatorname{{\rm tr.deg}}}

\usepackage{ulem}

\usepackage{color}

\newcommand{\Frac}{\operatorname{{\rm frac}}}
\newcommand{\Cl}{\operatorname{{\rm Cl}}}
\newcommand{\lcm}{\operatorname{{\rm lcm}}}
\newcommand{\setspec}[2]{\big\{\,#1\, \mid \,#2\, \big\}}
\renewcommand{\emptyset}{\varnothing}
\newcommand{\rank}{\operatorname{{\rm rank}}}
\newcommand{\haut}{\operatorname{{\rm ht}}}

\newcommand{\Jeul}{\EuScript{J}}

\newenvironment{enumerata}%
{\begin{enumerate}[label=\rm(\alph*), ref=\rm\alph*]}
{\end{enumerate}}

\newcommand{\Pgoth}{{\ensuremath{\mathfrak{P}}}}
\newcommand{\Integ}{\ensuremath{\mathbb{Z}}}
\newcommand{\Qgoth}{{\ensuremath{\mathfrak{Q}}}}
\newcommand{\isom}{\cong}

\theoremstyle{plain}

\newtheorem{theorem}{Theorem}[section]
\newtheorem{proposition}[theorem]{Proposition}
\newtheorem{lemma}[theorem]{Lemma}
\newtheorem{corollary}[theorem]{Corollary}

\theoremstyle{definition}

\newtheorem{example}[theorem]{Example}
\newtheorem{remark}[theorem]{Remark}

\newtheorem{notation}[theorem]{Notation}

\newtheorem{subnothing*}[sub]{}

\theoremstyle{remark}

\begin{document}

\makeatletter	   
\makeatother     


\title{Generalizations of Samuel's Criteria \\ for a Ring to be a Unique Factorization Domain}
\author{Daniel Daigle \and Gene Freudenburg \and Takanori Nagamine}
\date{\today} 
\thanks{The work of the first author was supported by grant 04539/RGPIN/2015 from NSERC Canada.}
\thanks{The work of the third author was supported by JSPS Overseas Challenge Program for Young Researchers (No. 201880243) and JSPS KAKENHI Grant Numbers JP18J10420 and JP20K22317.}

\pagestyle{plain}  

\begin{abstract} We give several criteria for a ring to be a UFD, including generalizations of some criteria due to P.\,Samuel. 
These criteria are applied to construct, for any field $k$, (1) a $\Z$-graded non-noetherian rational UFD of dimension 3 over $k$, and (2) $k$-affine rational UFDs defined by trinomial relations. 
\end{abstract}
\maketitle

\section{Introduction} Let $A$ be a unique factorization domain (UFD). This paper considers ring extensions of the following two types. 
\begin{itemize}
\item [(i)] $A[x]$ where $ax=b$ for relatively prime $a,b\in A\setminus\{ 0\}$ such that $aA$ and $aA+bA$ are prime ideals.
\item [(ii)] $A[x]$ where $A$ has a $\Z$-grading and $x^n=F$ for some positive integer $n$ and homogeneous prime $F\in A$ with degree relatively prime to $n$.
\end{itemize}
In his 1964 treatise on UFDs, Samuel \cite{Samuel.64} studied each of these kinds of extension. 

In case (i) Samuel demonstrated that, if $A$ is noetherian, then $A[x]$ is a UFD (\cite{Samuel.64}, Proposition\,7.6). 
We show that the noetherian condition can be weakened to a local condition, namely, if the hypotheses of (i) are satisfied and $\bigcap_{i\ge 1}(aA+bA)^i=(0)$, then $A[x]$ is a UFD
({\it Theorem\,\ref{First-Criterion}}). 

In case (ii) Samuel considered rings of the form $A=R[X_1,\hdots, X_n]$, a polynomial ring over a UFD $R$, where $A$ is $\Z$-graded by positive weights over $R$. He showed that, if either $n$ is congruent to 1 modulo $\deg F$, or every finitely generated projective $R$-module is free, then $A[x]$ is a UFD (\cite{Samuel.64}, Theorem\,8.1). We show, more generally, that $A[x]$ is a UFD whenever the conditions of (ii) hold ({\it Theorem\,\ref{Third-Criterion}}). 

After some preliminaries, {\it Section\,\ref{UFD}} gives these and other new criteria for a ring to be a UFD, and applies them
to certain families of affine algebras over a field $k$ which motivated our work to generalize Samuel's criteria in the first place. 
In {\it Section~\ref{examples}}, we construct a $\Z$-graded non-noetherian UFD over $k$ with Krull dimension 3 and quotient field $k^{(3)}$, the field of rational functions in three variables over $k$. 
This example is similar to one given in \cite{David.73}, but the existence of a $\Z$-grading allows for a much simpler proof.
{\it Section\,\ref{app-two}} introduces a family of $k$-affine rings defined by trinomial relations. {\it Theorem\,\ref{trinomial}} shows that the rings in this family are UFDs. Such rings appear in the classification of affine factorial $k$-varieties which admit a torus action of complexity one. For an algebraically closed field $k$ of characteristic zero, these rings were studied in 
\cite{Hausen.Herppich.Suss.11}.

The Bourbaki volume \cite{Bourbaki.72} claims that $A[x]$ is a UFD whenever the conditions of (i) above hold. 
However, this assertion is wrong. We construct a counterexample in {\it Section 6}. 


\section{Preliminaries}\label{preliminaries}

All rings are commutative with unity and any domain is understood to be an integral domain. For the ring $A$, $A^*$ is the group of units of $A$, 
$\dim A$ is the Krull dimension of $A$ and, for an ideal $I\subset A$, ${\rm ht}(I)$ is the height of $I$. If $f\in A$ then $A_f=S^{-1}A$ 
where $S=\{ f^n\, |\, n\in\N\}$.
For the integer $n\ge 0$, $A^{[n]}$ is the polynomial ring in $n$ variables over $A$, and $A^{[\pm n]}$ is the ring of Laurent polynomials in $n$ variables over $A$. 
When $A$ is an integral domain, ${\rm frac}(A)$ is its field of fractions, and if $A\subset B$ are integral domains, then ${\rm tr.deg}_AB$ is the transcendence degree of ${\rm frac}(B)$ over ${\rm frac}(A)$. For the field $K$, $K^{(n)}$ denotes the field of fractions of the polynomial ring $K^{[n]}$. 

\subsection{Krull Domains and Divisor Class Groups}  \label {c09i203sdnao99}
Given an integral domain $A$, we write $\bP(A)$ for the set of height $1$ prime ideals of $A$.
Krull domains can be characterized as follows (\cite{Bourbaki.72}, Chap.~VII, \S~1, n$^\circ$ 7, Thm~4):

\begin{theorem} \label {coiv0w309gc}
A ring $A$ is a Krull domain if and only if it is an integral domain satisfying each of the following three conditions. 
\begin{enumerata}

\item For each $\pgoth \in \bP(A)$, $A_\pgoth$ is a DVR.

\item $A = \bigcap_{ \pgoth \in \bP(A) }{ A_\pgoth }$

\item For each $x \in A \setminus \{0\}$, $\setspec{ \pgoth \in \bP(A) }{ x \in \pgoth }$ is a finite set.

\end{enumerata}
\end{theorem}

Let $A$ be a Krull domain, $K = \Frac(A)$, 
and $\Div(A)$ the free abelian group on the set $\bP(A)$.
Elements of $\Div(A)$ are formal sums $\sum_{\pgoth \in \bP(A)} n_\pgoth \pgoth$ with $n_\pgoth \in \Z$ for all
$\pgoth \in \bP(A)$ and $n_\pgoth = 0$ for almost all $\pgoth$.
For each $\pgoth \in \bP(A)$, $A_\pgoth$ is a discrete valuation ring and we denote by $v_\pgoth : K^* \to \Z$
the corresponding normalized valuation (i.e., the valuation with $v_\pgoth( K^* ) = \Z$).
Given $x \in K^*$, the set $\setspec{ \pgoth \in \bP(A) }{ v_\pgoth(x) \neq 0 }$ is finite by {\it Theorem~\ref{coiv0w309gc}(c)};
so
$$
\textstyle
\div_A : K^* \to \Div(A), \quad \div_A(x) = \sum_{\pgoth \in \bP(A)} v_\pgoth(x) \pgoth \ \ (x \in K^*)
$$
defines a group homomorphism.
The elements of $\Div(A)$ are called the {\bf divisors} of $A$ and those of 
$\text{\rm Prin}(A) = \setspec{ \div_A(x) }{ x \in K^* }$ 
are called the {\bf principal divisors} of $A$.
The quotient group $\Cl(A) = \Div(A) /\text{\rm Prin}(A)$  is the {\bf divisor class group} of $A$.
See \cite{Matsumura.86}, \S\,20.

\begin{proposition}\label{Krull-intersection} {\rm ( \cite{Samuel.64}, Cor.\,(a) to Prop.\,4.1)} The intersection of a finite number of Krull domains (within a common field) is a Krull domain. 
\end{proposition} 

\begin{proposition}\label{Krull-localization} {\rm (\cite{Samuel.64}, Prop.\,4.2 and Thm.\,6.3)} If $A$ is a Krull domain and $S\subset A\setminus\{ 0\}$ is a multiplicatively closed set, then $S^{-1}A$ is a Krull domain. If $S$ is generated by prime elements of $A$, then ${\rm Cl}(S^{-1}A)\cong {\rm Cl}(A)$. 
\end{proposition}

\begin{proposition}\label{Krull-descent} 
Let $A$ be an integral domain, let $p_1,\hdots ,p_n\in A\setminus\{ 0\}$ be primes, and let 
$S\subset A\setminus\{ 0\}$ be the multiplicatively closed set generated by $p_1,\hdots ,p_n$. 
\begin{itemize}
\item [{\bf (a)}] $A=S^{-1}A\cap A_{(p_1)}\cap\cdots \cap A_{(p_n)}$
\smallskip
\item [{\bf (b)}] Assume that $A_{(p_i)}$ is a DVR, $1\le i\le n$. If $S^{-1}A$ is a Krull domain, then $A$ is a Krull domain. 
\end{itemize}
\end{proposition} 

\begin{proof}
Let $x \in S^{-1}A \cap A_{(p_1)}\cap\cdots \cap A_{(p_n)}$.
Since $x \in S^{-1}A$, we may write $x = a/s$ where $a \in A$ and $s \in S$ are chosen
in such a way that for each $i$ satisfying $p_i \mid s$, we have $p_i \nmid a$.
If $s \notin A^*$ then there exists $i$ such that  $p_i \mid s$ (so $p_i \nmid a$);
since $x \in A_{(p_i)}$, we have $x = a_i/s_i$ for some $a_i \in A$ and $s_i \in A \setminus (p_i)$;
thus $a s_i = a_i s$, so $p_i \mid (a s_i)$ where $p_i \nmid a$ and $p_i \nmid s_i$, a contradiction.
So $s \in A^*$ and hence $x \in A$. This proves (a).
Part (b) follows from part (a) and {\it Proposition\,\ref{Krull-intersection}}.
\end{proof}

\begin{proposition}\label {ckjpq9wxdiqpows} {\rm (\cite{Fossum.73}, Prop.\,6.1)} 
A ring $A$ is a UFD if and only if it is a Krull domain satisfying $\Cl(A)=0$.
\end{proposition}

\subsection{Nagata's Criterion}
Let $A$ be an integral domain. It is well known that, if $A$ is a UFD, then every localization of $A$ is a UFD. 
A partial converse is given by Nagata. Recall that an integral domain is {\bf atomic} if every nonzero element factors as a finite product of irreducible elements.
Recall also that every Krull domain is atomic. 

\begin{theorem}\label{Nagata}  Let $A$ be an integral domain and $S\subset A$ a multiplicatively closed set generated by a set of prime elements, $0\not\in S$.
Assume that $S^{-1}A$ is a UFD. If $A$ is an atomic domain, then $A$ is a UFD. 
\end{theorem}

Nagata's original formulation assumed $A$ to be noetherian (\cite{Nagata.57}, Lemma 2). 
Samuel extended it to the case $A$ is a Krull domain (\cite{Samuel.64}, Corollary to Theorem 6.3). 
Then Kaplansky generalized it to the case where $A$ satisfies the ACCP (\cite{Kaplansky.74}, Theorem 177). 
The version of the criterion stated above, together with an elementary proof, is due to Cohn (\cite{Cohn.77}, \S\,11.3, Theorem 5). 

\subsection{Samuel's Criterion} 
Recall that, in an integral domain $A$, elements $a,b\in A\setminus\{ 0\}$ are {\bf relatively prime} if $aA\cap bA=abA$. 
\begin{theorem}\label{Samuel-lemma}  {\rm (\cite{Samuel.64}, Proposition 7.6.)} Assume that $A$ is an integral domain and $a,b\in A\setminus\{ 0\}$ are relatively prime.
Let $A[X]\cong A^{[1]}$ and $A'=A[X]/(aX-b)$, and consider the subring $A[b/a]$ of ${\rm frac}(A)$. 
\begin{itemize}
\item [{\bf (a)}] The kernel of the $A$-surjection $A[X]\to A[b/a]$, $X\mapsto b/a$ equals $(aX-b)$. Consequently, $(aX-b)$ is a prime ideal of $A[X]$ and
$A'\cong A[b/a]$.
\item [{\bf (b)}] If $A$ is a noetherian normal domain and $aA$ and $aA+bA$ are prime ideals, then $A'$ is a noetherian normal domain and ${\rm Cl}(A')\cong {\rm Cl}(A)$. 
\item [{\bf (c)}] If $A$ is a noetherian UFD and $aA$ and $aA+bA$ are prime ideals, then $A'$ is a noetherian UFD. 
\end{itemize}
\end{theorem}
The statement of this result in Samuel's paper does not assume $a,b\ne 0$ but this was clearly an oversight, since the result is not generally true if $a=0$ or $b=0$. 
Part (a) of this theorem is somewhat stronger than Samuel's original formulation; it states the full consequences of Samuel's proof. 
We also need:
\begin{lemma}\label{rel-prime} Assume that $A$ is an integral domain and $f,g\in A\setminus\{ 0\}$ are relatively prime. 
\begin{itemize}
\item [{\bf (a)}] $f^m$ and $g^n$ are relatively prime for each pair of integers $m,n\ge 1$.
\item [{\bf (b)}] $f$ and $g+fh$ are relatively prime for each $h\in A$ such that $g+fh\ne 0$.
\item [{\bf (c)}] Given $h\in A$, if $gh\in fA$, then $h\in fA$.
\item [{\bf (d)}] Let $B$ be an integral domain containing $A$. If $B$ is a free $A$-module, then $f$ and $g$ are relatively prime in $B$. 
\end{itemize}
\end{lemma}
The reader can easily verify parts (a), (b) and (c) of this lemma; part (d) is \cite{Freudenburg.17}, Lemma 2.46. Note that part (d) includes the case $B=A^{[n]}$ for some $n\ge 1$. 

\subsection{Prime Avoidance}
The following lemma represents an instance of the Prime Avoidance Lemma.\footnote{See for
example: www.math.lsa.umich.edu/\!$\sim$hochster/615W17/supDim.pdf }

\begin{lemma}\label{coprime} Given $n\in\N$ and $a_1,\hdots ,a_n,b,c\in\Z$ with $\gcd (a_1,\hdots ,a_n,b,c)=1$, there exist $m_1,\hdots ,m_n\in\Z$ with 
$\gcd (c,b+m_1a_1+\cdots +m_na_n)=1$.
\end{lemma}

\begin{proof} The lemma is clearly true if $|c|=1$, so assume that $|c|\ge 2$. 
Let $(a_1,\hdots ,a_n)=d\Z$, and let $p_1, \dots, p_s$ be the distinct prime factors of $c$ ($s\ge 1$).
Suppose that $b+d\Z\subset p_1\Z\cup\cdots\cup p_s\Z$. 
Let $\{ p_1,\hdots ,p_s\}=S_1\cup S_2$, where elements of $S_1$ divide $b$, and elements of $S_2$ do not divide $b$. By hypothesis, 
$S_1\ne\emptyset$. Set $\Sigma_i=\bigcup_{p\in S_i}p\Z$ ($i=1,2$) and 
let $\lambda=b+d\sigma$, where $\sigma$ is the product of the elements of $S_2$. Then $b\in\Sigma_1$, and by hypothesis, either $\lambda\in\Sigma_1$ or $\lambda\in\Sigma_2$. 
If $\lambda\in\Sigma_1$, then $d\sigma\in\Sigma_1$ implies $d\in\Sigma_1$, a contradiction, since $\gcd (d,b,c)=1$. If $\lambda\in\Sigma_2$, then $b\in\Sigma_2$, also a contradiction. 
Therefore, $b+d\Z$ is not contained in $p_1\Z\cup\cdots\cup p_s\Z$. 
Pick an element $b + \sum_{i=1}^n m_i a_i$ not in $\bigcup_{j=1}^s p_j\Z$, where $m_i\in\Z$. Then
$\gcd\big( c, b + \sum_{i=1}^n m_i a_i \big) = 1$.
\end{proof}

\subsection{$\Z$-Gradings}

\begin{lemma}\label{characterization-bis}
Let $A$ be an integral domain with a non-trivial $\Z$-grading and let $\alpha$ be a nonzero homogeneous element of $A$. 
There exist homogeneous elements $f\in \alpha A \setminus\{0\}$ and $w\in A_f^{\ast}$ such that 
\[
A_f=(A_f)_0[w,w^{-1}]\cong (A_f)_0^{[\pm 1]}
\]
where $(A_f)_0$ denotes the subring of $A_f$ of degree zero elements. 
\end{lemma}

\begin{proof} 
Let $d = \gcd\setspec{ i \in \Integ }{ A_i \neq 0}$ and note that $d>0$ because the grading is not trivial.
There exist homogeneous elements $a,b \in A \setminus \{0\}$ such that $\deg(a) - \deg(b) = d$.
Define $f = \alpha ab \in \alpha A \setminus\{0\}$ and $w =a/b$, then $w \in A_f^*$ and $w$ is homogeneous of degree $d$.
Given any homogeneous $x \in A_f$, we have $\deg(x) = id$ for some $i \in \Z$, so 
$x/w^i \in (A_f)_0$ and hence $x \in (A_f)_0[w,w^{-1}]$, showing that $A_f = (A_f)_0[w,w^{-1}]$.
Since $\deg(w)>0$, $w$ is algebraically independent over $(A_f)_0$,
so $(A_f)_0[w,w^{-1}]\cong (A_f)_0^{[\pm 1]}$.
\end{proof}


\section{Criteria for a Ring to be a UFD}\label{UFD}


\subsection{First Criterion}

Let $A$ be an integral domain, and let $(a,b)\in A^2$. Define the set
\[
P(A,(a,b))=\text{\rm set of prime elements $p$ of $A$ satisfying  $a\in pA$ and $pA+bA\ne A$.}
\]
The pair $(A,(a,b))$ is said to satisfy {\bf condition} $\mathcal{P}$ if it satisfies each of the following four conditions.
\begin{itemize}
\item [$\mathcal{P}$(i):] $a,b$ are nonzero and relatively prime, and $a$ is either a unit or a product of primes of $A$.
\smallskip
\item [$\mathcal{P}$(ii):] $pA+bA$ is a prime ideal for every $p\in P(A,(a,b))$.
\smallskip
\item [$\mathcal{P}$(iii):] $q\not\in pA+bA$ for every non-associate pair $p,q\in P(A,(a,b))$.
\smallskip
\item [$\mathcal{P}$(iv):] $\bigcap_{i\ge 0}(pA+bA)^i=(0)$ for every $p\in P(A,(a,b))$. 
\end{itemize}
Note that if $a$ is a power of a prime then condition $\mathcal{P}$(iii) is satisfied, since there are no
non-associate pairs $p,q\in P(A,(a,b))$ in this case. 
The following generalizes {\it Theorem\,\ref{Samuel-lemma}}.  

\begin{theorem}\label{First-Criterion}
Let $A$ be a Krull domain, let $(a,b)\in A^2$, and define the ring
\[
A'=A[X]/(aX-b)
\]
where $A[X]\cong A^{[1]}$.  Assume that $(A,(a,b))$ satisfies condition $\mathcal{P}$.
\begin{itemize}

\item [{\bf (a)}] $A'$ is a Krull domain and ${\rm Cl}(A')\cong {\rm Cl}(A)$.

\item [{\bf (b)}] $A'$ is a UFD if and only if $A$ is a UFD.

\end{itemize}
\end{theorem} 

Several preliminaries are needed for the proof of this theorem.

Recall that if $W$ is an ideal of a ring $A$ and $t \in A$, one defines the ideal $(W:t)$ by:
\[
(W:t) = \setspec{ x \in A }{ tx \in W }
\]
\begin{notation}
Any triple $(b,s,t)$ of elements of a ring $A$ determines an ideal $W(b,s,t)$ of $A$, defined as follows.
First, define sequences $J_n$ and  $W_n$ of ideals of $A$ ($n\ge 0$) by setting $W_0 = A$,
$J_0 = (W_0:t)=A$, and:
$$
W_{i} = b J_{i-1} + s^iA, \quad J_{i} = (W_{i} : t) \qquad \text{for all $i\ge1$.}
$$
Then $W_{i+1}\subset W_i$ and $J_{i+1}\subset J_i$ for each $i\ge 0$.  Set $W(b,s,t) = \bigcap_{i\ge0} W_i$.
\end{notation}

\begin{lemma} \label {pckjnv2039d992}
Let $A$ be a domain, $a,b \in A\setminus\{0\}$ relatively prime, and let $s,t \in A$ be such that
$a=st$ and $s,t$ are relatively prime. Let $A[X] = A^{[1]}$ and $I = (aX-b)A[X]$.
Then the following are equivalent:
\begin{enumerata}

\item $\bigcap_{i\geq 0}(s^i A[X]+I)= I$,

\item  $A \cap \bigcap_{i\geq 0}(s^i A[X]+I)= 0$.

\end{enumerata}
Moreover, $A \cap \bigcap_{i\geq 0}(s^i A[X]+I)= W(b,s,t)$.
\end{lemma}

\begin{proof}
Since $I \cap A = 0$, it is clear that (a) implies (b).
To prove the converse, suppose that (b) is true and consider $f \in \bigcap_{i\geq 0}(s^i A[X]+I)$.
By the division algorithm, there exist $N\ge0$, $Q \in A[X]$ and $r \in A$ such that $a^N f = (a X-b)Q+r$.
Note that $r = -(a X-b)Q+ a^N f$ where both $aX-b$ and $f$ belong to $\bigcap_{i\geq 0}(s^i A[X]+I)$; 
so $r \in A \cap \bigcap_{i\geq 0}(s^i A[X]+I)$. Since (b) is true, we have $r=0$, so $a^N f = (a X-b)Q$;
since $a, b$ are relatively prime, {\it Lemma\,\ref{rel-prime}} implies
$a^N \mid Q$ in $A[X]$, so $f \in I$, proving that (a) is true.
This shows that (a) and (b) are equivalent.

To prove that $A \cap \bigcap_{i\geq 0}(s^i A[X]+I) \subset W(b,s,t)$,
consider an element $r \in A \cap \bigcap_{i\geq 0}(s^i A[X]+I)$.
Fix $n>0$ and let us prove that $r \in W_n$. 

We have $r = (a X-b) U + s^n V$ for some $U,V \in A[X]$.
Write $U = \sum_{i=0}^\infty u_i X^i$ and $V = \sum_{i=0}^\infty v_i X^i$ where $u_i,v_i \in A$ for all $i\ge0$
and $u_i=0=v_i$ for $i\gg0$. Then
\begin{equation} \label {pc0vowe9f}
\text{$-b u_0 + s^n v_0 = r$ \qquad and \qquad $a u_{i-1} - b u_i + s^n v_i = 0$ for all $i > 0$.}
\end{equation}
We claim:
\begin{equation} \label {Ee8d98d125g0jf}
s^i \mid u_i \quad \text{for all $i \in \{0, \dots, n\}$.}
\end{equation}
We proceed by induction on $i$, the case $i=0$ being obvious.
Assume that $i$ is such that $0 < i \le n$ and $s^{i-1} \mid u_{i-1}$;
then \eqref{pc0vowe9f} implies that $s^i \mid b u_i$, so $s^i \mid u_i$.
This proves \eqref{Ee8d98d125g0jf}.
Define $(u_0', \dots, u_n') \in A^{n+1}$ by $u_i' = u_i / s^{i}$ ($0 \le i \le n$).
Dividing the second part of \eqref{pc0vowe9f} by $s^{i}$ gives
\begin{equation} \label {cjvuweuipsx09dz8xrr}
t u_{i-1}' - b u_i' + s^{n-i} v_i = 0 \quad \text{for all $i \in \{ 1, \dots, n \}$.}
\end{equation}
By descending induction on $i$, we shall now prove that
\begin{equation}  \label {lcjv98w7b5d4dfjoe}
u'_{i} \in J_{n-i-1} \qquad \text{for all $i \in \{0, \dots, n-1\}$.}
\end{equation}
Since $u'_{n-1} \in A = J_0$, the case $i=n-1$ of \eqref{lcjv98w7b5d4dfjoe} is true.
If $i$ is such that $0<i\le n-1$ and $u'_{i} \in J_{n-i-1}$, then \eqref{cjvuweuipsx09dz8xrr} gives
$t u_{i-1}' = b u_i' - s^{n-i} v_i \in b J_{n-i-1} + s^{n-i}A = W_{n-i}$, so $u'_{i-1} \in (W_{n-i} : t) = J_{n-i}$.
This proves \eqref{lcjv98w7b5d4dfjoe}. It follows that $u_0 = u_0' \in J_{n-1}$ and consequently that
$r = -b u_0 + s^n v_0 \in b J_{n-1} + s^n A = W_n$.
Since $n$ is arbitrary, $r \in  W(b,s,t)$.
This shows that $A\cap\bigcap_{i\geq 0}(s^i A[X]+I) \subset W(b,s,t)$.

For the reverse inclusion, consider an element $r \in  W(b,s,t)$.
Fix $n>0$, and let us prove that $r \in  s^n A[X]+I$.
It is convenient to define $J_{-1}=A$, and to note that the relations $W_{i} = b J_{i-1} + s^iA$ and $J_{i} = (W_{i} : t)$
are also valid for $i=0$.
We first prove that assertions $\Pscr(0), \dots, \Pscr(n)$ are true,
where for each $j\in\{0,\dots,n\}$ we define
{\it
\begin{enumerate}
\item[$\Pscr(j)$:\ ]  There exist $(u_0', \dots, u_j'), (v_0, \dots, v_j) \in A^{j+1}$ satisfying:
\begin{itemize}

\item[$(a)$] $r = -b u_0' + s^n v_0$,

\item[$(b)$] if $j>0$ then $t u_{i-1}' - b u_i' + s^{n-i} v_i = 0$ for all $i \in \{ 1, \dots, j \}$,

\item[$(c)$] $u_j' \in J_{n-j-1}$.

\end{itemize}
\end{enumerate}}
Since  $r \in  W(b,s,t)$,
we have $r \in W_n = b J_{n-1} + s^n A$
and so we can choose $u_0' \in J_{n-1}$ and $v_0 \in A$ such that $r = -b u_0' + s^n v_0$.
So $\Pscr(0)$ is true.
Consider $j \in \{1, \dots, n\}$ such that $\Pscr(j-1)$ is true.
Then we have $(u_0', \dots, u_{j-1}'), (v_0, \dots, v_{j-1}) \in A^{j}$ satisfying the above conditions.
Since $u_{j-1}' \in J_{n-j} = (W_{n-j}:t)$, we have $t u_{j-1}' \in W_{n-j} = b J_{n-j-1} + s^{n-j}A$,
so we can choose $u_j' \in J_{n-j-1}$ and $v_j \in A$ such that $t u_{j-1}' - b u_j' + s^{n-j}v_j = 0$.
Then $(u_0', \dots, u_j'), (v_0, \dots, v_j) \in A^{j+1}$ satisfy the requirements of $\Pscr(j)$.
By induction, it follows that $\Pscr(0), \dots, \Pscr(n)$ are true.
Since $\Pscr(n)$ is true, there  exist $(u_0', \dots, u_n'), (v_0, \dots, v_n) \in A^{n+1}$
satisfying $r = -b u_0' + s^n v_0$ and \eqref{cjvuweuipsx09dz8xrr}.
Then 
$$
\textstyle
(a X-b) \big( \sum_{i=0}^{n-1} u_i' s^i X^i \big) + s^n \big( -b u_n' X^n + \sum_{i=0}^n v_i X^i \big) = r ,
$$
so $r \in  s^n A[X]+I$ and we are done.
\end{proof}

\begin{lemma} \label {dpf9i2p03wd0dX3od8tfq5}
Let $A$ be a ring and $b,s,t \in A$.
\begin{enumerata}

\item If $t \in \bigcap_{i\ge0} (bA + s^iA)$ then $W(b,s,t) = \bigcap_{i\ge0} (bA + s^iA)$.

\item If $bA+sA+tA=A$, or if $b,s,t$ is an $A$-regular sequence, then  $W(b,s,t) = \bigcap_{i\ge0} (bA+sA)^i$.

\end{enumerata}
\end{lemma}

\begin{proof}
(a) If $t \in \bigcap_{i\ge0} (bA + s^iA)$ then one can prove by induction that $W_i= bA+s^iA$ and $J_i=A$ for all $i\ge0$,
so the conclusion follows.

(b) Let $\Jeul = bA + sA$. Observe that if the condition
\begin{equation}  \label {9126e6126qwydhfn}
(\Jeul^i : t) = \Jeul^i \quad \text{for all $i\ge0$}
\end{equation}
is satisfied, then it follows by induction that $W_i = \Jeul^i = J_i$ for all $i\ge0$,
from which we get the desired conclusion. So it's enough to prove \eqref{9126e6126qwydhfn}.
If $bA+sA+tA=A$ then $\Jeul^i + tA=A$ for all $i\ge0$, so \eqref{9126e6126qwydhfn} is true and we are done.
So we may assume that $b,s,t$ is an $A$-regular sequence.
Also, the case $i=0$ of \eqref{9126e6126qwydhfn} is trivially true, and the case $i=1$ is an immediate consequence of the
fact that  $b,s,t$ is an $A$-regular sequence.
By part (i) of Thm 27 on page 98 of \cite{Matsumura.80}, it suffices to prove that $b,s$ is an $A$-quasiregular sequence.
The fact that $b,s$ is $A$-quasiregular follows from part (ii) of the same Theorem
together with the fact that $b,s$ is an $A$-regular sequence.
\end{proof}

\begin{proof}[Proof of Theorem\,\ref{First-Criterion}] Assume that $(A,(a,b))$ satisfies $\mathcal{P}$. 
Then $A'$ is an integral domain, by {\it Theorem\,\ref{Samuel-lemma}}, and $A$ is a subring of $A'$. 
We may assume that $a \notin A^*$, otherwise the claim is trivial. So, by $\mathcal{P}$(i), $a$ is a nonempty product of primes.
Let $P=P(A,(a,b))$, and let $Q$ be the set of primes $p \in A$ such that $a \in pA$ (note that $P \subset Q$).
If $p\in Q$ then: 
\[
A'/pA' \cong A[X]/(aX-b,p)\cong A/(pA+bA)[X]\cong \left( A/(pA+bA)\right)^{[1]}
\]
Therefore, if $p \in P$, then $p$ is a prime element of $A'$; and if $p \in Q \setminus P$, then $p$ is a unit of $A'$.
	
Let $S \subset T$ be the multiplicative sets of $A$ generated by $P$ and $Q$ respectively.
Since $a \in T$, we have $T^{-1}A'=T^{-1}A$;
since $Q \setminus P \subset {A'}^*$, we have $S^{-1} A' = T^{-1} A'$;
thus $S^{-1}A'=T^{-1}A$. Therefore, by {\it Proposition\,\ref{Krull-localization}}, $S^{-1}A'$ is a Krull domain. 

Since $A$ is a Krull domain, the set $\{ pA\, |\, p\in P\}$ is finite; consequently, $\{ pA'\, |\, p\in P\}$ is finite. By {\it Proposition\,\ref{Krull-descent}}, if $A^{\prime}_{(p)}$ is a DVR for each $p\in P$, then $A'$ is a Krull domain. 
We thus need to show:
\begin{equation} \label {pc9ivjp2093wdj}
\text{$\bigcap_{i\geq 0} p^iA'=(0)$ for every $p\in P$}
\end{equation}
Let $I=(aX-b)A[X]$.
For each $p \in P$, define elements $s(p)$ and $t(p)$ of $A$ by the conditions:
$$
\text{$a = s(p) \cdot t(p)$,\ \  $s(p)$ is a power of $p$,\ \  and $t(p) \notin pA$}
$$
Given $p \in P$, the condition $\bigcap_{i\geq 0} p^iA'= (0)$ is equivalent to 
$\bigcap_{i\geq 0} s(p)^iA'= (0)$, which is equivalent to $\bigcap_{i\geq 0}(s(p)^iA[X]+I)= I$,
which (by {\it Lemma~\ref{pckjnv2039d992}}) is equivalent to $W(b,s(p),t(p))= (0)$.
So condition \eqref{pc9ivjp2093wdj} is equivalent to:
\begin{equation} \label {ckjJo0vp9w3syys8a7aiuwiHhf8e}
\text{$W(b, s(p), t(p) ) = (0)$ for every $p\in P$}
\end{equation}
We show that condition $\mathcal{P}$(iii)
implies that $W(b, s(p), t(p) ) = \bigcap_{i\ge0} (p A + b A)^i$ for every $p\in P$. By condition $\mathcal{P}$(iv), this suffices to prove (\ref{pc9ivjp2093wdj}). 

Given $p\in P$, let $s = s(p)$ and $t = t(p)$.
In view of {\it Lemma~\ref{dpf9i2p03wd0dX3od8tfq5}(b)} and of the fact that
\[
\bigcap_{i\ge0} (p A + b A)^i = \bigcap_{i\ge0} (s A + b A)^i
\]
it suffices to show that $bA+sA+tA=A$ or $b,s,t$ is an $A$-regular sequence.
Assume that $bA+sA+tA \neq A$ and let us prove that $b,s,t$ is $A$-regular.
Since $s=p^e$ for some $e\ge1$, it suffices to show that $b,p,t$ is $A$-regular (see \cite{Matsumura.80}, Thm~26, p.~96).
Note that $\pgoth := bA+pA$ is a prime ideal, because $p \in P$ and $\mathcal{P}$(ii) is true.
Since $t \notin A^*$ (because $bA+sA+tA \neq A$), and since $a$ is a product of primes,
 we have $t = q_1 \cdots q_n$ for some $q_1, \dots, q_n \in Q$.
If $q_i \in P$ then $q_i \notin \pgoth$ by condition $\mathcal{P}$(iii);
if $q_i \in Q \setminus P$ then $q_i A + \pgoth \supset q_i A + bA=A$, so again $q_i \notin \pgoth$. So $t \notin \pgoth$.
Clearly, $A \xrightarrow{b} A$ (multiplication by $b$) is injective.
Since $p$ is prime and $p \nmid b$, $A/bA \xrightarrow{p} A/bA$ is injective.
Since $\pgoth$ is a prime ideal and $t \notin \pgoth$, $A/\pgoth \xrightarrow{t} A/\pgoth$ is injective.
So  $b,p,t$ is $A$-regular, and equation (\ref{pc9ivjp2093wdj}) is confirmed. 

Therefore, $A^{\prime}_{(p)}$ is a DVR for each $p\in P$, and $A'$ is a Krull domain. In addition, by {\it Proposition\,\ref{Krull-localization}}:
\[
{\rm Cl}(A)\cong {\rm Cl}(T^{-1}A)\cong {\rm Cl}(S^{-1}A')\cong {\rm Cl}(A')
\]
This proves assertion (a), and assertion (b) follows immediately from (a) and {\it Proposition\,\ref{ckjpq9wxdiqpows}}. 
\end{proof}

\begin{corollary}\label{noetherian} Let $A$ be a noetherian UFD, and let $(a,b)\in A^2$. Define the ring
\[
A'=A[X]/(aX-b)
\]
where $A[X]\cong A^{[1]}$. If $(A,(a,b))$ satisfies $\mathcal{P}${\rm (i)} and $\mathcal{P}${\rm (ii)}, then $A'$ is a noetherian UFD. 
\end{corollary}

\begin{proof} As in the proof of {\it Theorem\,\ref{First-Criterion}}, $A'$ is an integral domain, each $p\in P(A,(a,b))$ is prime in $A'$, and $S^{-1}A'=T^{-1}A$. Since $A'$ is noetherian, it follows by Nagata's Criterion that $A'$ is a UFD.
\end{proof}


\subsection{Second Criterion} 

\begin{theorem} \label {class-group} 
Consider Krull domains $A \subset B$ such that $B$ is finitely generated as an $A$-module.
There exist group homomorphisms
$$
\bar i : \Cl(A) \to \Cl(B) \quad \text{and}  \quad \bar N : \Cl(B) \to \Cl(A)
$$
such that $\bar N \circ \bar i :  \Cl(A) \to \Cl(A)$ is the map $\gamma \mapsto n\gamma$ with $n =[ \Frac(B) : \Frac(A) ]$. 
\end{theorem}

Bourbaki gives a proof of the special case of this theorem where $A$ and $B$ are assumed to be noetherian
(see \cite{Bourbaki.72}, Chap.~VII, \S~4, n$^\circ$ 8).
Being unable to find a reference for the theorem in its general form, we give a proof.
In fact, our proof is essentially the one given in Bourbaki.
The following pargraphs are preparatory. We use the notations introduced in paragraph \ref{c09i203sdnao99}.

Consider Krull domains $A \subset B$ such that $B$ is finitely generated as an $A$-module.
Let $K=\Frac(A)$ and $L=\Frac(B)$, and let $n = [ L:K ]$.
Then for each $\Pgoth \in \bP(B)$, we have $\Pgoth \cap A \in \bP(A)$.
Moreover, for each $\pgoth \in \bP(A)$
the set $\setspec{ \Pgoth \in \bP(B) }{ \Pgoth \cap A = \pgoth }$ is nonempty and (by Thm~\ref{coiv0w309gc}(c)) finite,
and $\setspec{ B_{\Pgoth} }{ \Pgoth \in \bP(B) \text{ and } \Pgoth \cap A = \pgoth }$
is the set of all valuation rings $R$ of $L$ satisfying $R \cap K = A_\pgoth$.
Caution: if $\Pgoth \in \bP(B)$ and $\pgoth = \Pgoth \cap A$ then
the map $v_\Pgoth : L^* \to \Z$ is not an extension of the map $v_\pgoth : K^* \to \Z$,
because both $v_\pgoth$ and $v_\Pgoth$ are normalized by convention (i.e., $v_\pgoth(K^*) = \Z$ and $v_\Pgoth(L^*)=\Integ$);
however, there exists a positive integer $e( \Pgoth/\pgoth )$ such that
the valuation $v_\Pgoth' = \frac1{e(\Pgoth/\pgoth)} v_\Pgoth : L^* \to \frac1{e(\Pgoth/\pgoth)}\Z$
is an extension of $v_\pgoth$.
The number $e( \Pgoth/\pgoth )$ is the ramification index of $v_\Pgoth'$ over $v_\pgoth$,
and we have $v_{\Pgoth}(x) = e(\Pgoth/\pgoth) v_{\pgoth}(x)$ for all $x \in K^*$.
We shall also consider the residual degree $f(\Pgoth/\pgoth)$ of $v_\Pgoth'$ over $v_\pgoth$,
i.e., $f(\Pgoth/\pgoth)= [ \kappa(\Pgoth) : \kappa(\pgoth) ]$
where $\kappa(\Pgoth) = B_\Pgoth /\Pgoth B_\Pgoth$ and $\kappa(\pgoth) = A_\pgoth /\pgoth A_\pgoth$.
Also note that $\setspec{ v_{\Pgoth}' }{ \Pgoth \in \bP(B) \text{ and } \Pgoth \cap A = \pgoth }$
is a complete system of extensions of $v_\pgoth$ to $L$.

Given  $\pgoth \in \bP(A)$, we write $\Pgoth \mid \pgoth$ as an abbreviation for
``$\Pgoth \in \bP(B)$ and $\Pgoth \cap A = \pgoth$''.
Consider the norm $N_{L/K} : L^* \to K^*$. We claim:
\begin{gather}
\label {knxwoevjq3kwsi}
\text{for each $\pgoth \in \bP(A)$ we have}\ \ \textstyle \sum_{\Pgoth \mid \pgoth } e(\Pgoth/\pgoth) f(\Pgoth/\pgoth) = n, \\
\label {d9876344f8hrg7f6654} \textstyle
\text{for each $x \in L^*$ and $\pgoth \in \bP(A)$ we have
$v_\pgoth\big( N_{L/K}(x) \big) = \sum_{\Pgoth \mid \pgoth} f(\Pgoth/\pgoth) v_\Pgoth(x)$,}
\end{gather}
where $\sum_{ \Pgoth \mid \pgoth }$ means $\sum_{ \Pgoth \in X }$
with $X = \mbox{$\setspec{ \Pgoth \in \bP(B) }{ \Pgoth \cap A = \pgoth }$}$.
Indeed, let $\pgoth \in \bP(A)$. Then $A_\pgoth \subset B_\pgoth$ are Krull domains and
$B_\pgoth$ is finitely generated as an $A_\pgoth$-module.
Let $\qgoth$ denote the maximal ideal of $A_\pgoth$ (so $\bP(A_\pgoth)=\{\qgoth\}$).
The valuations $v_{\Qgoth}' = \frac1{e(\Qgoth/\qgoth)} v_\Qgoth$ with $\Qgoth \in \bP(B_\pgoth)$
constitute a complete system of extensions of $v_\qgoth=v_\pgoth$ to $L$.
Since $B_\pgoth$ is the integral closure of $A_\pgoth$ in $L$ and $B_\pgoth$ is finite over $A_\pgoth$, 
\cite{Bourbaki.72}, Chap.~VI, \S~8, n$^\circ$ 5, Thm~2 implies that 
\begin{equation*} 
\textstyle
\sum_{ \Qgoth \in \bP(B_\pgoth) } e(\Qgoth/\qgoth) f(\Qgoth/\qgoth) = n, 
\end{equation*}
and Cor.~3 of that Theorem implies that, for all $x \in L^*$,
\begin{equation*}   
\textstyle
v_\qgoth\big( N_{L/K}(x) \big)
= \sum_{\Qgoth \in \bP(B_\pgoth)} e(\Qgoth/\qgoth) f(\Qgoth/\qgoth) v_\Qgoth'(x)
= \sum_{\Qgoth \in \bP(B_\pgoth)} f(\Qgoth/\qgoth) v_\Qgoth(x).
\end{equation*}
The rule $\Qgoth \mapsto \Qgoth \cap B$ defines a bijection 
$\bP(B_\pgoth) \to \setspec{ \Pgoth \in \bP(B) }{ \Pgoth \cap A = \pgoth }$,
and for each $\Qgoth \in \bP(B_\pgoth)$ if we define $\Pgoth = \Qgoth \cap B$
then $e(\Qgoth/\qgoth) = e(\Pgoth/\pgoth)$, $f(\Qgoth/\qgoth) = f(\Pgoth/\pgoth)$, and $v_\Qgoth(x) = v_\Pgoth(x)$.
This proves \eqref{knxwoevjq3kwsi} and \eqref{d9876344f8hrg7f6654}.

We now proceed with the proof of {\it Theorem\,\ref{class-group}}. 
\begin{proof}
Let $A \subset B$ be Krull domains such that $B$ is a finitely generated $A$-module.
Let $K = \Frac(A)$,  $L = \Frac(B)$, and $n = [L : K]$.  Define the map
$$
\textstyle
i : \bP(A) \to \Div(B) , \quad i(\pgoth) = \sum_{ \Pgoth \mid \pgoth } e(\Pgoth/\pgoth) \Pgoth \ \ 
\text{(for each $\pgoth \in \bP(A)$)}.
$$
Extend $i$ to a group homomorphism $i: \Div(A) \to \Div(B)$.
By \cite{Bourbaki.72}, Chap.~VII, \S~1, n$^\circ$ 10, Prop.~14, we have
$$
i( \div_A(x) ) = \div_B( x ), \quad \text{for all $x \in K^*$}.
$$
So $i$ induces a group homomorphism $\bar i : \Cl(A) \to \Cl(B)$.

Define a map $N : \bP(B) \to \Div(A)$ by declaring that if $\Pgoth \in \bP(B)$ and $\pgoth = \Pgoth \cap A$ then
$N(\Pgoth) = f( \Pgoth / \pgoth ) \pgoth$. Extend $N$ to a homomorphism of groups $N : \Div(B) \to \Div(A)$.
We claim:
\begin{equation} \label {lkjcvpw3idlcsop}
N\big( \div_B(x) \big) = \div_A \big( N_{L/K}(x) \big) \quad \text{for each $x \in L^*$}.
\end{equation}
Indeed, let $x \in L^*$ and let $\phi : \bP(B) \to \bP(A)$ be the map $\Pgoth \mapsto \Pgoth \cap A$. Then
\begin{align*}
N\big( \div_B(x) \big) &= \textstyle N\big( \sum_{ \Pgoth \in \bP(B) } v_\Pgoth(x) \Pgoth \big)
= \sum_{ \Pgoth \in \bP(B) } v_\Pgoth(x) N( \Pgoth ) \\ 
&=\textstyle  \sum_{ \Pgoth \in \bP(B) } v_\Pgoth(x) f(\Pgoth/\phi(\Pgoth)) \phi( \Pgoth )
= \sum_{\pgoth \in \bP(A)} \sum_{ \Pgoth \mid \pgoth } v_\Pgoth(x) f(\Pgoth/\pgoth) \pgoth  \\ 
&=\textstyle  \sum_{\pgoth \in \bP(A)} v_\pgoth\big( N_{L/K}(x) \big) \pgoth 
= \div_A \big( N_{L/K}(x) \big) ,
\end{align*}
where the penultimate equality is \eqref{d9876344f8hrg7f6654}.
So \eqref{lkjcvpw3idlcsop} is true and, consequently, $N$ induces a group homomorphism $\bar N : \Cl(B) \to \Cl(A)$.

For each $\pgoth \in \bP(A)$, we have 
\[
N(i(\pgoth))
=N\big(  \sum_{ \Pgoth \mid \pgoth } e(\Pgoth/\pgoth) \Pgoth \big)
=\sum_{ \Pgoth \mid \pgoth } e(\Pgoth/\pgoth) N( \Pgoth )
=\sum_{ \Pgoth \mid \pgoth } e(\Pgoth/\pgoth) f(\Pgoth/\pgoth) \pgoth = n \pgoth
\]
by \eqref{knxwoevjq3kwsi}.
So $N \circ i :  \Div(A) \to \Div(A)$ is multiplication by $n$, and consequently
$\bar N \circ \bar i :  \Cl(A) \to \Cl(A)$ is multiplication by $n$.
\end{proof}

As a corollary to this theorem, we give the following descent property for integral extensions.

\begin{corollary}\label{Second-Criterion}
Consider  $S \supset R\subset T$ where $R$ is a Krull domain and $S,T$ are UFDs.
Assume that $S$ and $T$ are finitely generated $R$-modules and that the integers
\[
[\Frac(S):\Frac(R)] \quad {\rm and}\quad [\Frac(T):\Frac(R)]
\]
are relatively prime. Then $R$ is a UFD. 
\end{corollary}

\begin{proof}
By {\it Proposition~\ref{ckjpq9wxdiqpows}}, we have $\Cl(S)=0=\Cl(T)$ and it suffices to show that $\Cl(R)=0$.
Let $s=[\Frac(S):\Frac(R)]$ and $t=[\Frac(T):\Frac(R)]$ (so $\gcd(s,t)=1$).
By {\it Theorem~\ref{class-group}}, there exist group homomorphisms 
\[
f: \Cl(R)\to \Cl(S)\,\, ,\,\, g:\Cl(S)\to \Cl(R)\,\, ,\,\, F:\Cl(R)\to \Cl(T)\,\, ,\,\, G:\Cl(T)\to \Cl(R)
\]
such that $(g \circ f)(\gamma )=s\gamma$ and $(G \circ F)(\gamma )=t\gamma$ for every $\gamma\in \Cl(R)$.
Then $s\gamma =t\gamma =0$ (and hence $\gamma =0$) for every $\gamma\in \Cl(R)$, i.e., $\Cl(R)=0$ and $R$ is a UFD. 
\end{proof}


\subsection{Third Criterion} 
The following criterion generalizes Samuel \cite{Samuel.64}, Thm.\,8.1.  

\begin{theorem}\label{Third-Criterion} 
Let $A$ be an integral domain, $F\in A\setminus\{0\}$, and $c>0$ an integer,
and define the ring
\[
B=A[Z]/(Z^c-F)
\]
where $A[Z]\cong A^{[1]}$. 
Assume that there exists a $\Z$-grading of $A$ such that $F$ is homogeneous and $\gcd(c, \deg F) = 1$.
Then $B$ is an integral domain and ${\rm frac}(B)\cong {\rm frac}(A)$.
Moreover, if $F$ is prime in $A$ then the following hold.
\begin{itemize}

\item [{\bf (a)}] $A$ is a Krull domain if and only if $B$ is a Krull domain. 

\item [{\bf (b)}] Assume that $A, B$ are Krull domains.
Then $\Cl(B)$ is a direct summand of $\Cl(A) \oplus \Cl(A)$ and $\Cl(A)$ is a direct summand of $\Cl(B) \oplus \Cl(B)$.
Moreover, if one of $\Cl(A), \Cl(B)$ is finitely generated then $\Cl(A) \cong \Cl(B)$.

\item [{\bf (c)}] $A$ is a UFD if and only if $B$ is a UFD.

\end{itemize}
\end{theorem}

\begin{remark} \label {98543gfik3wi}
In {\it Theorem~\ref{Third-Criterion}},
the assertion  ${\rm frac}(B)\cong {\rm frac}(A)$ does not mean that the canonical map $A \to B$ extends to an isomorphism
of the fields of fractions. Actually, the canonical map extends to an embedding $\Frac(A) \to \Frac(B)$ that satisfies
$[\Frac(B) : \Frac(A)] = c$.
\end{remark}

\begin{remark} \label {p9cn293edco}
Let $\mathfrak{g}$ be the $\Z$-grading of $A$ in {\it Theorem~\ref{Third-Criterion}}
and let $\omega=\deg_{\mathfrak{g}}F$.
Extend the $\Z$-grading $c\,\mathfrak{g}$ of $A$ to a $\Z$-grading $\mathfrak{g}^{\prime}$ of $A[Z]$
by letting $Z$ be homogeneous of degree $\omega$.
Then $Z^c-F$ is $\mathfrak{g}'$-homogeneous and the quotient $B=A[Z]/(Z^c-F)=A[z]$ has the $\Z$-grading induced by $\mathfrak{g}^{\prime}$.
This $\Z$-grading of $B$ has the property that $z \in B$ is homogeneous of degree $\omega$.
\end{remark}

For the proof, we need the following facts. The proof of each of the first two lemmas is straightforward, and left to the reader.

\begin{lemma}\label{Laurent}
Let $R$ be an integral domain, $R[x,y]\cong R^{[2]}$, $\lambda\in R^*$ and $a,b\in\Z$ positive and relatively prime.
Then
\[
R[x,y]/(x^ay^b-\lambda) = R[z,z^{-1}]\cong R^{[\pm 1]}
\]
where $z=x^n/y^m$ for integers $m,n$ with $am+bn=1$.
\end{lemma}

\begin{lemma}\label{lemma1}
Let $(G,+)$ be an abelian group and $W = \bigoplus_{i \in G} W_i$ a $G$-graded ring.
Given any group homomorphism  $\theta : G \to W_0^*$, the map $\Phi_\theta : W \to W$ defined by
$\Phi_\theta\big( \sum_{i\in G} w_i \big) = \sum_{i\in G} \theta(i) w_i$ (where $w_i \in W_i$)
is an automorphism of $W$, both as a $G$-graded ring and as a $W_0$-algebra.
\end{lemma}

\begin{notation} \label {coiuv8273dcp0}
Let $G,H$ be abelian groups.
If $n \in \Z \setminus \{0\}$, we write $G \flouc{n} H$ as an abbreviation for:
\begin{quote}
{\it There exist group homomorphisms
$G \xrightarrow{\phi} H \xrightarrow{\psi} G$ such that $\psi \circ \phi : G \to G$ is the multiplication by $n$.}
\end{quote}
We write $G \flouc{*} H$ as an abbreviation for:
\begin{quote}
{\it For each $r \in \Z\setminus\{0\}$, there exists $n \in \Z\setminus\{0\}$ such that $\gcd(n,r)=1$ and $G \flouc{n} H$.}
\end{quote}
We write $\rank_0(G)$ for the torsion-free rank of $G$, i.e., the dimension of the $\Q$-vector space $\Q \otimes_\Z G$.
\end{notation}

\begin{lemma}  \label {kcjvo983w5esuiek}
Let $G,H$ be abelian groups.
\begin{itemize}

\item [{\bf (a)}] 
If there exists $n \in \Z \setminus\{0\}$ such that $G \flouc{n} H$, then $\rank_0(G) \le \rank_0(H)$.

\item [{\bf (b)}] If there exist $m,n \in \Z \setminus\{0\}$ such that $G \flouc{m} H$, $G \flouc{n} H$ and $\gcd(m,n)=1$,
then $G$ is a direct summand of $H \oplus H$.
In particular, if $G \flouc{*} H$ then $G$ is a direct summand of $H \oplus H$.

\item [{\bf (c)}]
Assume that there exist relatively prime integers $n,r\neq0$ satisfying $G \flouc{n} H$ and $rG=0$.
Then $G$ is a direct summand of $H$.

\item [{\bf (d)}] Assume that one of $G,H$ is finitely generated.
If $G \flouc{*} H$ and $H \flouc{*} G$, then $G \cong H$.

\end{itemize}
\end{lemma}

\begin{proof}
(a)
Consider homomorphisms
$G \xrightarrow{\phi} H \xrightarrow{\psi} G$ such that $\psi \circ \phi : G \to G$ is the multiplication by $n$.
Then we have $\Q \otimes_\Z G \xrightarrow{\phi'} \Q \otimes_\Z H \xrightarrow{\psi'} \Q \otimes_\Z G$
where $\psi' \circ \phi' : \Q \otimes_\Z G \to \Q \otimes_\Z G$ is the multiplication by $n$;
so $\psi' \circ \phi'$ is an automorphism of $\Q \otimes_\Z G$, so $\phi'$ is injective and  $\rank_0(G) \le \rank_0(H)$.

(b) For each $i \in \{m,n\}$, let $G \xrightarrow{\phi_i} H \xrightarrow{\psi_i} G$ be homomorphisms
such that $\psi_i \circ \phi_i : G \to G$ is the multiplication by $i$.
Choose $a,b \in \Z$ such that $am+bn=1$ and define
$G \xrightarrow{\phi'} H \oplus H \xrightarrow{\psi'} G$ by declaring that
$\phi'(g) = (\phi_m(g), \phi_n(g))$ for all $g \in G$ and $\psi'(x,y) = a \psi_m(x) + b \psi_n(y)$ for all $x,y \in H$.
Then $\psi' \circ \phi'$ is the identity map of $G$, so $G$ is a direct summand of $H \oplus H$.

(c) 
Consider $G \xrightarrow{\phi} H \xrightarrow{\psi} G$ such that $\psi \circ \phi : G \to G$ is the multiplication by $n$.
Choose $u,v \in \Z$ such that $un+vr=1$ and define $\psi' : H \to G$ by $\psi'(y) = u\psi(y)$ for all $y \in H$.
Then $x = unx$ for all $x \in G$, so $\psi' \circ \phi$ is the  identity map of $G$, so $G$ is a direct summand of $H$.

(d) The hypothesis together with (b) implies that 
$G$ is a direct summand of $H \oplus H$ and
$H$ is a direct summand of $G \oplus G$; so both $G,H$ are finitely generated.
By part (a), we have $\rank_0(G) = \rank_0(H)$, so it suffices to show that $T(G) \cong T(H)$ (where $T(G)$ denotes the torsion
subgroup of $G$).  Since $G \mapsto T(G)$ is an additive functor,
we have $T(G) \flouc{n} T(H)$ for each $n$ satisfying $G \flouc{n} H$, so we have $T(G) \flouc{*} T(H)$,
and by symmetry $T(H) \flouc{*} T(G)$. 
Since $T(G)$ is a finite group, there exists an integer $r>0$ satisfying $rT(G)=0$.
By (c), $T(G)$ is a direct summand of $T(H)$.
By symmetry, $T(H)$ is a direct summand of $T(G)$. Since $T(G),T(H)$ are finite, they are isomorphic.
\end{proof}

\begin{proof}[Proof of Theorem\,\ref{Third-Criterion}]
Let $\omega = \deg(F)$ with respect to the given $\Z$-grading of $A$.
With no loss of generality, we may assume that $\omega\ge0$.
If $\omega=0$ then $c=1$ and $B \cong A$, in which case the result is trivial. So we shall assume throughout that $\omega>0$.

Let us first show that $Z^c-F$ is an irreducible element of $K[Z]$, where $K = \Frac(A)$.
In view of \cite{Lang.93}, Theorem 9.1, it suffices to show that $F$ cannot be written in the form $F = n\xi^d$
with $\xi \in K$, $n \in \Integ\setminus\{0\}$ and $d>1$ a divisor of $c$. 
By contradiction, suppose that $F$ is written in that form. Write $\xi = U/V$ where $U,V \in A$. Then $F V^d = n U^d$, which implies
that $d \mid \omega$. Since $d \mid c$, $d>1$ and $\gcd(c,\omega)=1$, we have a contradiction.
So $Z^c-F$ is an irreducible element of $K[Z]$.
Now let $\tau$ be an element of an algebraic closure of $K$ satisfying $\tau^c=F$; then $Z^c-F$ is the minimal polynomial of $\tau$ over $K$.
Consider the surjective $A$-homomorphism $\phi : A[Z] \to A[ \tau ]$ defined by $\phi(Z)=\tau$.  Clearly, $(Z^c-F)A[Z] \subset \ker\phi$.
Consider an element $G(Z) \in \ker\phi$.
By the division algorithm, there exist $Q(Z), r(Z) \in A[Z]$ such that $G(Z) = Q(Z) (Z^c-F) + r(Z)$ and $\deg_Z r(Z) < c$.
Then $r(\tau)=0$.
Since $Z^c-F$ is the minimal  polynomial of $\tau$ over $K$, we have $r(Z)=0$, so $G(Z) = Q(Z) (Z^c-F) \in (Z^c-F)A[Z]$.
Thus $\ker\phi =  (Z^c-F)A[Z]$ and consequently $B = A[Z] / (Z^c-F)$ is isomorphic to $A[\tau]$. So $B$ is a domain.
Moreover, the above argument shows that $[ \Frac(B) : \Frac(A)] = c$.

By {\it Lemma\,\ref{characterization-bis}}, there exist homogeneous $f\in F A \setminus \{0\}$
and $x\in (A_f)^*$ (where the $\Z$-grading of $A$ is extended to $A_f$) such that $A_f=(A_f)_0[x,x^{-1}]\cong (A_f)_0^{[\pm 1]}$. 
Interchanging $x$ and $x^{-1}$ if necessary, we have $\deg x = e$, where we define $e = \gcd\setspec{ i \in \Z }{ A_i \neq 0 }$.
Since $F$ is homogeneous of degree $\omega$, there exists $\kappa\in (A_f)_0$ such that $F=\kappa x^{\omega/e}$ in $A_f$. 
Since $F \mid f$, we have $F \in A_f^*$; so $\kappa \in A_f^*$ and consequently $\kappa$ is a unit of $(A_f)_0$.
Note that the canonical homomorphism $\pi :A[Z]\to B$ is injective on $A$. Let $z=\pi (Z)$, so that $B=A[z]$.
Let $m,n\in\Z$ be such that $cm+(\omega/e) n=1$. 
By {\it Lemma\,\ref{Laurent}} we see that
\begin{eqnarray*}
B_f &=& A_f[Z]/(Z^c-F) \\
&=& (A_f)_0[x,x^{-1},Z]/(Z^c-\kappa x^{\omega/e}) \\
&=& (A_f)_0[x,x^{-1},Z]/((x^{-1})^{\omega/e}Z^c-\kappa ) \\
&=& (A_f)_0[x,y,y^{-1}] \\
&=& (A_f)_0[y,y^{-1}]\cong (A_f)_0^{[\pm 1]} \cong A_f
\end{eqnarray*}
where $y  = z^n x^m$. 
So $B_f \cong A_f$ and in particular $\Frac(B) \cong \Frac(A)$.
This proves the first assertion.

Since $B/zB \cong A/FA$, it is clear that $F$ is prime in $A$ if and only if $z$ is prime in $B$.
Until the end of the proof, we assume that  $F$ is prime in $A$ (and $z$ is prime in $B$).

(a) 
Choose integers $s,t>0$ such that $sc \equiv 1 \pmod{\omega}$, $tc \equiv 1 \pmod{\omega}$, and $\gcd(s,t)=1$
(for instance, first choose $s>0$ such that $sc \equiv 1 \pmod{\omega}$, and let $t = s + |\omega|$).
Let $U$ and $V$ be indeterminates over $B$ and define
\[
S=B[U]/(U^s-z)=B[u]
\quad \text{and} \quad
T=B[V]/(V^t-z)=B[v] .
\]
Note that $\gcd(s,\omega)=1$ and that (by {\it Remark}~\ref{p9cn293edco})
there is a $\Z$-grading of $B$ such that $z$ is homogeneous of degree $\omega$;
this allows us to apply the first part of the proof to $S$ and conclude that $S$ is a domain and $[ \Frac(S) : \Frac(B) ] = s$.
Similarly, $T$ is a domain and $[ \Frac(T) : \Frac(B) ] = t$.
Also define $\Omega = B[Q] / (Q^{st}-z) = B[q]$ where $B[Q] = B^{[1]}$; since $\gcd(st,\omega)=1$, $\Omega$ is a domain.
Then $S[X]/(X^t-u) = \Omega = T[Y]/(Y^s-v)$ and we have the commutative diagram of integral domains
\[
\xymatrix@R=2pt{
&& T \ar[rd]\\
A \ar[r] & B \ar[ru] \ar[rd] & & \Omega \\
&& S \ar[ru]
}
\]
where all homomorphisms are injective.
The fact that $\gcd(s,t)=1$ implies that
$$
S \cap T = B \quad \text{(taking the intersection in $\Omega$).}
$$
Indeed, this follows by considering that $\Omega = B[q]$ is a free $B$-module with basis $\{ 1, q, \dots, q^{st-1}\}$,
and that $S = B[q^t]$ and $T = B[q^s]$ are free $B$-modules with bases $\{ q^{ti} \}_{i=0}^{s-1}$ and $\{ q^{si} \}_{i=0}^{t-1}$ respectively.

Since $F=z^c=u^{sc}$, we see that:
\[
S=A[u]\cong A[U]/(U^{sc}-F)
\]
Let $W = A[U,U^{-1}]\cong A^{[\pm 1]}$.
The $\Z$-grading of $A$ extends to a $\Z$-grading of $W$ in which $U$ is homogeneous of degree $0$ and $W_0 = A_0[U,U^{-1}]$.
Since $sc \equiv 1 \pmod{\omega}$, there exists $d \in \Z$ such that $sc=d\omega +1$.
Consider the group homomorphism $\theta : \Z \to W_0^*$, $\theta(i)= U^{di}$
and the corresponding automorphism $\Phi = \Phi_\theta : W \to W$ defined in {\it Lemma~\ref{lemma1}}.
Then 
\[
\Phi(U^{sc}-F) = U^{sc} - U^{d\omega}F = U^{d\omega}(U-F)
\]
where $U^{d\omega} \in W^*$,
so $\Phi$ maps the principal ideal $(U^{sc}-F)W$ onto the principal ideal $(U-F)W$. So $\Phi$ induces an isomorphism $\bar\Phi$ in:
\[
\xymatrix{
S_u \ar[r]^-{\cong} & W/(U^{sc}-F) \ar[r]^-{\cong}_-{\bar\Phi} &  W/(U-F) \ar[r]^-{\cong} &  A_F .
}
\]
Therefore, $S_u\cong A_F$.
Since $t$ satisfies $tc \equiv 1 \pmod{\omega}$, the same argument shows that $T_v \cong A_F$.

Assume that $A$ is a Krull domain.
Since  $S_u\cong A_F$,  $S_u$ is a Krull domain by {\it Proposition\,\ref{Krull-localization}}.
In addition, since $S/uS\cong A/FA$, $u$ is prime in $S$. 
As an $A$-module, we have:
\[
S=A\oplus Au\oplus\cdots\oplus Au^{sc-1}
\]
Choose $h\in\bigcap_{i\in\N}u^iS$, and write $h=h_0+h_1u+\cdots +h_{sc-1}u^{sc-1}$ for $h_j\in A$. By hypothesis, 
\[
h\in u^{sci}S=F^iS=F^iA\oplus F^iAu\oplus\cdots\oplus F^iAu^{sc-1}
\]
for each $i\in\N$. Therefore, $h_j\in F^iA$ for each $j=0,\hdots ,sc-1$ and each $i\in\N$. Since $A$ is a Krull domain, $h_j=0$ for each $j=0,\hdots ,sc-1$, so $h=0$.
Therefore, $\bigcap_iu^iS=(0)$. Since $u$ is prime in $S$, it follows that $S_{(u)}$ is a DVR.
By {\it Proposition\,\ref{Krull-descent}}, $S$ is a Krull domain. 
So $T$ is a Krull domain by the same argument, and $B = S \cap T$ is a Krull domain by  {\it Proposition\,\ref{Krull-intersection}}.
This shows that if $A$ is a Krull domain then so is $B$.

Conversely, assume that $B$ is a Krull domain.
Since $S=B[U]/(U^s-z)=B[u]$ and $\gcd(s,\omega)=1$, the above argument shows that $S$ is a Krull domain.
As $u$ is prime in $S$, we get $\bigcap_j u^j S = 0$.
Using $u^{sc} = F$, we see that $\bigcap_i F^i A \subset \bigcap_j u^j S = 0$, so $A_{(F)}$ is a DVR.
We also know that  $A_F$ is a Krull domain, because $S_u\cong A_F$.
By {\it Proposition\,\ref{Krull-descent}}, $A$ is a Krull domain. 
This proves (a), but let us also observe that,
since $S_u \cong A_F$ where $u$ is prime in $S$ and $F$ is prime in $A$,
{\it Proposition}~\ref{Krull-localization} implies that $\Cl(S) \cong \Cl(S_u) \cong \Cl(A_F) \cong \Cl(A)$.

(b) Assume that $A,B$ are Krull domains.
The proof of (a) shows that if $s \in \Z$ satisfies $s>0$ and $sc \equiv 1 \pmod{\omega}$
then $S=B[U]/(U^s-z)$ is a Krull domain and a finite $B$-module such that $[ \Frac(S) : \Frac(B)] = s$ and $\Cl(S) \cong \Cl(A)$.
By {\it Theorem}~\ref{class-group}, there exist group homomorphisms $\Cl(B) \to \Cl(S) \to \Cl(B)$
whose composition $\Cl(B) \to \Cl(B)$ is multiplication by $s$.
This means that $\Cl(B) \flouc{s} \Cl(S)$ (see {\it Notation}~\ref{coiuv8273dcp0}),
and since $\Cl(S) \cong \Cl(A)$, we have in fact
shown that $\Cl(B) \flouc{s} \Cl(A)$ for every integer $s>0$ satisfying $sc \equiv 1 \pmod{\omega}$.
It follows that
\begin{equation}  \label {Ocklo2893bdf9qp0ctvw5e}
\Cl(B) \flouc{*} \Cl(A) .
\end{equation}
Indeed, let $r \in \Z \setminus\{0\}$.
Choose any $s_0 \in \Z$ such that $s_0c \equiv 1 \pmod{\omega}$.
Since $\gcd(\omega,s_0,r)=1$, {\it Lemma}~\ref{coprime} implies that there exists $m \in \Z$ satisfying $\gcd(r,s_0+m\omega)=1$;
then for $N$ large enough, the number $s = s_0 + m\omega + N \cdot |r| \cdot \omega$ satisfies $s>0$, $\gcd(s,r)=1$ and $sc \equiv 1 \pmod{\omega}$,
so $\gcd(s,r)=1$ and $\Cl(B) \flouc{s} \Cl(A)$.
This proves \eqref{Ocklo2893bdf9qp0ctvw5e}. Note that \eqref{Ocklo2893bdf9qp0ctvw5e} says that we have $\Cl(B) \flouc{*} \Cl(A)$
{\it for any pair of rings $(A,B)$ that satisfies the hypotheses of assertion} (b).
Let $s \in \Z$ be such that $s>0$ and $sc \equiv 1 \pmod{\omega}$ and consider $S=B[U]/(U^s-z)$;
then, in view of {\it Remark\,\ref{p9cn293edco}}, the pair $(B,S)$ satisfies  the hypotheses of assertion (b), so \eqref{Ocklo2893bdf9qp0ctvw5e} implies that $\Cl(S) \flouc{*} \Cl(B)$.
It follows that  $\Cl(A) \flouc{*} \Cl(B)$, because $\Cl(S) \cong \Cl(A)$ was noted at the beginning of the proof of (b). Thus:
$$
\Cl(A) \flouc{*} \Cl(B) \quad \text{and} \quad \Cl(B) \flouc{*} \Cl(A) .
$$
Now assertion (b) follows from {\it Lemma}~\ref{kcjvo983w5esuiek}.

Assertion (c) is an immediate consequence of (a), (b) and  {\it Proposition\,\ref{ckjpq9wxdiqpows}}.
\end{proof}

\begin{corollary}\label{Pham-Brieskorn} Let $R$ be a UFD, $a_1,\hdots ,a_n\in\N\setminus\{ 0\}$ and $R[X_1, \hdots ,X_n]\cong R^{[n]}$, where $n\ge 3$. 
Assume that one of the following holds:
\begin{enumerate}

\item [{\rm (1)}] $n \ge 4$ and $\gcd (a_n,a_1\cdots a_{n-1})=1$; 

\item [{\rm (2)}] $n=3$ and $a_1,a_2,a_3$ are pairwise relatively prime.
\end{enumerate}
Then the ring $R[X_1,\hdots ,X_n]/(X_1^{a_1}+\cdots +X_n^{a_n})$ is a UFD.
\end{corollary}

\begin{proof} Let $A=R[X_1, \hdots ,X_{n-1}]$ and define $F\in A$ by $F=-X_1^{a_1}-\cdots -X_{n-1}^{a_{n-1}}$. 
The conditions on the $a_i$ imply that  $F$ is irreducible in ${\rm frac}(R)[X_1, \hdots ,X_{n-1}]$, and it follows easily that $F$ is irreducible in $A$. 
Let $\omega= \lcm(a_1, \dots, a_{n-1})$ and let $A$ have the $\Z$-grading over $R$ for which $X_i$ is homogeneous of degree $\omega/a_i$. 
Then $F$ is homogeneous of degree $\omega$ and $\gcd (\omega ,a_n)=1$. 
{\it Theorem~\ref{Third-Criterion}} implies that 
\[
R[X_1,\hdots ,X_n]/(X_1^{a_1}+\cdots +X_n^{a_n})\cong A[Z]/(Z^{a_n}-F)
\]
is a UFD.
\end{proof}


\subsection{Fourth Criterion}

\begin{theorem}\label{Fourth-Criterion}
Let $A$ be a $\Z$-graded integral domain, let $a,b\in A\setminus\{ 0\}$ be relatively prime, and let $n>0$ be such that $\gcd(n, \deg b - \deg a) = 1$. 
Let $A[Z] = A^{[1]}$.
\begin{itemize}
\item [{\bf (a)}] $(aZ^n-b)$ is a prime ideal of $A[Z]$. 
\item [{\bf (b)}] Assume that $A$ is a noetherian UFD.  If $b$ is prime in $A$ and $(A,(a,b))$ satisfies $\mathcal{P}${\rm (ii)},
then $B:=A[Z] / (a Z^n - b)$ is a UFD and ${\rm frac}(B)\cong {\rm frac}(A)$. 
\end{itemize}
\end{theorem}

\begin{proof} Extend the $\Z$-grading to $A[X]\cong A^{[1]}$ so that $X$ is homogeneous and $\deg X=\deg b-\deg a$. Then $aX-b$ is homogeneous
and the ring $A'=A[X]/(aX-b)$ is $\Z$-graded.
Write $A'=A[x]$, where $x$ is the canonical image of $X$ in $A'$ (in particular, $x\ne 0$).
Then $x$ is homogeneous and $\deg x=\deg b-\deg a$.
It follows by {\it Theorem\,\ref{Samuel-lemma}} that $A'$ is an integral domain.
It is also clear that ${\rm frac}(A')={\rm frac}(A)$. We have:
\[
B\cong A[X,Z]/(aX-b,Z^n-X)\cong A'[Z]/(Z^n-x)
\]
Since $\gcd(n, \deg(x))=1$,  {\it Theorem~\ref{Third-Criterion}} implies that $B$ is an integral domain. This proves part (a). 

Assume that the hypotheses of part (b) hold. Then 
$(A,(a,b))$ satisfies $\mathcal{P}{\rm (i)}$ and, since $A$ is a UFD,
$A'$ is a UFD by {\it Corollary\,\ref{noetherian}}. 
Since $A'/xA' \cong A/bA$, we see that $x$ is a prime element of $A'$.
So {\it Theorem~\ref{Third-Criterion}} implies that  $B\cong A'[Z]/(Z^n-x)$ is a UFD, and that 
${\rm frac}(B)\cong {\rm frac}(A')={\rm frac}(A)$. 
So part (b) is proved. 
\end{proof}


\subsection{Fifth Criterion}

\begin{theorem}\label{Fifth-Criterion} Let $A=\bigoplus_{i\in\Z}A_i$ be a $\Z$-graded integral domain and let
$F\in A_{\omega} \setminus \{ 0 \}$, $\omega\in\Z$. 
Let $A[Z_1,\hdots ,Z_n]\cong A^{[n]}$ for $n\ge 1$, and let $e_1,\hdots e_n\ge 1$ be integers such that 
$\gcd(e_1,\hdots ,e_n,\omega )=1$.
\begin{itemize}
\item [{\bf (a)}] $(Z_1^{e_1}\cdots Z_n^{e_n}-F)$ is a prime ideal of $A[Z_1,\hdots ,Z_n]$
\item [{\bf (b)}] Define:
\[
 B=A[Z_1,\hdots ,Z_n]/(Z_1^{e_1}\cdots Z_n^{e_n}-F)
 \]
If $A$ is a noetherian UFD and $F$ is prime in $A$, then $B$ is a UFD and ${\rm frac}(B)\cong {\rm frac}(A)^{(n-1)}$. 
\end{itemize}
\end{theorem}

\begin{proof} We may assume that $n>1$, otherwise both parts of the claim follow from  {\it Theorem~\ref{Third-Criterion}}.

By {\it Lemma~\ref{coprime}}, there exist $m_1,\hdots ,m_{n-1} \in \Z$ so that:
\[
\gcd(e_n,\omega-(m_1e_1+\cdots +m_{n-1}e_{n-1}))=1 
\]
Extend the $\Z$-grading of $A$ to a $\Z$-grading of  $R:= A[Z_1, \dots, Z_{n-1}]$ by declaring that $Z_i$ is homogeneous of degree $m_i$, noting that $R$ is an integral domain. 
Define $a,b\in R$ by $a = Z_1^{e_1} \cdots Z_{n-1}^{e_{n-1}}$ and $b = F$. Then $a$ and $b$ are nonzero elements of $R$ with $aR\cap bR=abR$, and  
$\deg b=\omega$  and $\deg a=m_1e_1+\cdots +m_{n-1}e_{n-1}$.
Since $\gcd (e_n,\deg b-\deg a)=1$, {\it Theorem~\ref{Fourth-Criterion}(a)} implies that $(aZ_n^{e_n}-b)=(Z_1^{e_1}\cdots Z_n^{e_n}-F)$ is a prime ideal of $A[Z_1,\hdots ,Z_n]$. This proves part (a). 

Assume that the hypotheses of part (b) hold. 
Note that $R$ is a noetherian UFD. It is easy to check that $b$ is prime in $R$, $b \nmid a$ in $R$, and for each $i \in \{ 1, \dots, n-1\}$, $(Z_i,b)$ is a prime ideal of $R$.
So  {\it Theorem~\ref{Fourth-Criterion}(b)} implies that $B$ is a UFD, and that ${\rm frac}(B)\cong {\rm frac}(R)\cong {\rm frac}(A)^{(n-1)}$. 
\end{proof}


\section{Application I: Rational UFDs of Dimension Three}\label{examples}

\subsection{A Family of Three-dimensional Affine UFDs}
The following lemma generalizes Lemma 10.15 in \cite{Fossum.73}, and Lemma 2 in \cite{Daigle.Freudenburg.01a}. 

\begin{lemma} \label {prime}
Let $D$ be an integral domain, $n\ge 0$,  $u_1,\hdots ,u_n\in D^{\ast}$, $R_n=D[Z_0,\hdots ,Z_n]\cong D^{[n+1]}$, 
and $a_1,\hdots ,a_n$, $b_1,\hdots ,b_n$ positive integers such that $\gcd (a_i,b_1\cdots b_i)=1$ for each $i \in \{1, \dots, n\}$. Then
\[
I_n:=(u_1Z_1^{a_1}+Z_0^{b_1},\hdots , u_nZ_n^{a_n}+Z_{n-1}^{b_n})
\]
is a prime ideal of $R_n$ and $Z_n \notin I_n$.
\end{lemma}

\begin{proof}
We proceed by induction on $n$, the case $n=0$ being clear: $I_0=(0)\subset R_0 = D[Z_0]$. 
Let $n\ge 1$ and assume that $I_{n-1}$ is a prime ideal of $R_{n-1}=D[Z_0,\hdots ,Z_{n-1}]$ and that $Z_{n-1} \notin I_{n-1}$.
Define a $\Z$-grading of $R_{n-1}$ over $D$ for which $Z_i$ is homogeneous of degree $b_1\cdots b_ia_{i+1}\cdots a_{n-1}$, $0\le i\le n-1$.
Then the quotient ring $A:=R_{n-1}/I_{n-1}$ is a $\Z$-graded integral domain.

Let $F\in A$ be the image of $Z_{n-1}^{b_n}$;
since $Z_{n-1} \notin I_{n-1}$ and  $I_{n-1}$ is prime, we have  $Z_{n-1}^{b_n} \notin I_{n-1}$ and hence $F \neq 0$; note
that $\deg F=b_1\cdots b_n$. By hypothesis, $\gcd (a_n,\deg F)=1$. Therefore, by {\it Theorem~\ref{Third-Criterion}(a)}, the ring
$A[Z]/(Z^{a_n}-u_n^{-1}F)\cong R_n/I_n$ is an integral domain, so $I_n$ is prime. 
If $Z_n \in I_n$ then the image of $Z$ in $A[Z]/(Z^{a_n}-u_n^{-1}F)$ is zero, which is not the case because $F\neq 0$.
So $Z_n \notin I_n$ and the proof is complete.
\end{proof}

\begin{theorem}\label{threefold}
Let $K$ be a noetherian UFD, $n \in \N$, $K[Z_0,\hdots ,Z_{n+1}]\cong K^{[n+2]}$, $u_i,v_i\in K^*$
and $a_i,b_i$ positive integers such that $\gcd (a_i,b_1\cdots b_i)=1$, $1\le i\le n$. 
Define the ring
\[
A_n = K[Z_0,\hdots ,Z_{n+1}]/ (f_iZ_{i+1}+u_iZ_i^{a_i}+v_iZ_{i-1}^{b_i})_{1\le i\le n} 
\]
where $f_1,\hdots ,f_n\in K\setminus\{ 0\}$ and
the set of prime factors of $f_i$ in $K$ is the same for all $i=1, \dots, n$.
Then $A_n$ is a UFD and ${\rm frac}(A_n)\cong {\rm frac}(K[Z_0,Z_1])\cong ({\rm frac}\, K)^{(2)}$.
\end{theorem}

\begin{proof} Since $f_iZ_{i+1}+u_iZ_i^{a_i}+v_iZ_{i-1}^{b_i}=v_i(v_i^{-1}f_iZ_{i+1}+u_iv_i^{-1}Z_i^{a_i}+Z_{i-1}^{b_i})$, we may assume each $v_i=1$. 
Let $\pi_n :K[Z_0,\hdots ,Z_{n+1}]\to A_n$ be the standard surjection. The restriction of $\pi_n$ to $K$ is injective. To see this, observe that the ideal
$(f_iZ_{i+1}+u_iZ_i^{a_i}+Z_{i-1}^{b_i})_{1\le i\le n}$ of  $K[Z_0,\hdots ,Z_{n+1}]$ is included in $(Z_0, \dots, Z_{n+1})$,
and $(Z_0, \dots, Z_{n+1}) \cap K = \{0\}$. So $K\subset A_n$. 

Let $P_K$ be the set of prime elements of $K$  dividing some (hence all) $f_i$.
If $P_K = \emptyset$ then $f_1, \dots, f_n$ are units of $K$ and consequently $A_n \cong K[Z_0,Z_1] = K^{[2]}$, in which
case the theorem is true. So we may assume, throughout, that $P_K \neq \emptyset$.
Given $\kappa \in P_K$ and $m \in \{0,\dots,n\}$, define the ring:
\[
Q(m,\kappa)=(K/\kappa K)[Z_0,\hdots ,Z_m]/(\bar{u}_iZ_i^{a_i}+Z_{i-1}^{b_i})_{1\le i\le m} 
\]
where $\bar{u}_i$ is the class of $u_i$ in $K/\kappa K$. By {\it Lemma~\ref{prime}}, $Q(m,\kappa )$ is an integral domain.
Note that $\kappa \neq 0$ in $A_m$ because $\pi_m$ is injective on $K$. Moreover,
\[
A_m/\kappa A_m = K[Z_0,\hdots ,Z_{m+1}]/(\kappa , f_iZ_{i+1}+u_iZ_i^{a_i}+Z_{i-1}^{b_i})_{1\le i\le m}\cong Q(m,\kappa )[Z_{m+1}]\cong Q(m,\kappa )^{[1]} .
\]
Assuming that $m < n$, define $h_m\in A_m$ by $h_m=\pi_m (u_{m+1}Z_{m+1}^{a_{m+1}}+Z_m^{b_{m+1}})$; then
\[
A_m/(\kappa A_m+h_mA_m) = K[Z_0,\hdots ,Z_{m+1}]/(\kappa , u_iZ_i^{a_i}+Z_{i-1}^{b_i})_{1\le i\le m+1} \cong Q(m+1,\kappa ) .
\]
Also note that $h_m \notin \kappa A_m$.
Indeed, the image of $h_m$ in $A_m/\kappa A_m \cong Q(m,\kappa )[Z_{m+1}]\cong Q(m,\kappa )^{[1]}$ is not zero, because it has the form
$u Z_{m+1}^{a_{m+1}} + c$ with $u,c \in Q(m,\kappa )$, $u \neq 0$ and $a_{m+1}>0$. 
We have shown that  for each $\kappa \in P_K$, the following two statements are true:
\begin{gather}
\label {827365474gh394tj} \text{For each $m \in \{0,\dots,n\}$, $\kappa A_m$ is a nonzero prime ideal of $A_m$.}\\
\label {87cvbc355hap91} \text{For each $m \in \{0,\dots,n-1\}$,  $\kappa A_m+h_mA_m$ is a prime ideal of $A_m$ and  $h_m \notin \kappa A_m$.}
\end{gather}

By induction on $m$, we proceed to show $A_m$ is a UFD for all $m=0,\dots,n$.
Since $A_0=K[Z_0,Z_1]\cong K^{[2]}$, we have a basis for induction. 
Consider $m \in \{0,\dots,n-1\}$ such that $A_m$ is a UFD.
Let us prove the following assertions:
\begin{itemize}

\item[(i)] $f_{m+1}, h_m \in A_m \setminus\{0\}$, $f_{m+1}A_m \cap h_mA_m = f_{m+1}h_mA_m$, and $f_{m+1}$ is a product of primes in~$A_m$.

\item[(ii)] For each prime element $p$ of $A_m$ such that $f_{m+1}\in p A_m$, 
we have $p A_m + h_mA_m \in \Spec A_m$.
\end{itemize}
Since $P_K \neq \emptyset$, we have $f_{m+1} = \kappa_1 \cdots \kappa_r$ for some $\kappa_1, \dots, \kappa_r \in P_K$ and $r\ge1$.
By \eqref{827365474gh394tj}, each $\kappa_i$ is a prime element of $A_m$; so $f_{m+1} \neq 0$ in $A_m$ and $f_{m+1}$ is a product of primes in $A_m$.
Since $h_m \notin \kappa_1 A_m$ by \eqref{87cvbc355hap91}, we have $h_m \neq 0$ in $A_m$.
Actually, \eqref{87cvbc355hap91} gives $h_m \notin \kappa_i A_m$ for all $i = 1,\dots,r$, so $f_{m+1}$ and $h_m$ are relatively prime in $A_m$
and consequently $f_{m+1}A_m \cap h_mA_m = f_{m+1}h_mA_m$. This proves (i).
To prove (ii), consider a prime element $p$ of $A_m$ satisfying $f_{m+1}\in p A_m$.
Then, for some $i$, we have $p \mid \kappa_i$ in $A_m$.
Since $\kappa_i$ is a prime element of $A_m$,
we have $p = u\kappa_i$ for some $u \in A_m^*$, so $p A_m + h_mA_m = \kappa_i A_m + h_mA_m$.
We have $\kappa_i A_m + h_mA_m \in \Spec A_m$ by \eqref{87cvbc355hap91}, so (ii) is proved.

It follows from (i) and (ii) that $(A_m, (f_{m+1},h_m))$ satisfies the conditions $\mathcal{P}$(i) and $\mathcal{P}$(ii) stated just before Theorem \ref{First-Criterion}.
Since $A_{m+1}=A_m[Z_{m+2}]/(f_{m+1}Z_{m+2}+h_m)$, {\it Corollary\,\ref{noetherian}} implies that $A_{m+1}$ is a UFD.
By induction, we obtain that $A_0,\hdots ,A_n$ are UFDs. Since $A_n$ is in particular an integral domain, the assertion 
${\rm frac}(A_n)\cong {\rm frac}(K[Z_0,Z_1])\cong ({\rm frac}\, K)^{(2)}$ is now clear. 
\end{proof} 

Let $k[x]=k^{[1]}$ for a field $k$, $u_i,v_i\in k^*$, and let $p_1(x),\hdots ,p_n(x)\in k[x]\setminus\{ 0\}$ be such that each $p_i(x)$ has the same set of prime factors.
Define the affine $k$-algebra
\begin{equation}\label{kernel}
B_n=k[x][z_0,\hdots ,z_{n+1}]/(p_i(x)z_{i+1}+u_iz_i^{a_i}+v_iz_{i-1}^{b_i})_{1\le i\le n}
\end{equation}
where $a_1,\hdots ,a_n,b_1,\hdots ,b_n$ are positive integers such that $\gcd (a_i,b_1\cdots b_i)=1$ for each $i$.
Using $K=k[x]$ in {\it Theorem~\ref{threefold}}, it follows that, for each $n\ge 1$, $B_n$ is an affine rational UFD of dimension 3 over $k$.
\begin{proposition}\label{embed-dim}
If $p_i(x)\not\in k$ and $a_i,b_i\ge 2$ for all $i$, then the minimum number of generators of $B_n$ as a $k$-algebra is $n + 3$.  
\end{proposition}

\begin{proof} It suffices to prove the case where $k$ is algebraically closed.
Let $d$ be the minimum number of generators of $B_n$ over $k$. Then clearly $d\le n+3$. 
Set $X = {\rm Spec} (B_n)\subset \A_k^{n+3}$, affine $n$ space over $k$. For $1 \leq i \leq n$, let $f_i = p_i(x)z_{i+1}+u_iz_i^{a_i}+v_iz_{i-1}^{b_i}$. Let $J$ be the Jacobian matrix of $(f_1,\ldots ,f_n)$, namely:
	\begin{center}
	  $\displaystyle J = {\Bigg (}\frac{\partial f_i}{\partial x}, \frac{\partial f_i}{\partial z_j}{\Bigg )_{1 \leq i \leq n,\: 0 \leq j \leq n + 1}}$
	\end{center} 
Then $J$ is a matrix of size $n \times (n+3)$. For $1 \leq i \leq n$ and $0 \leq j \leq n + 1$, we have $\partial f_i/ \partial x =  p_i'(x)z_{i + 1}$ and: 
\[
\frac{\partial f_i}{\partial z_j} = 
	\begin{cases}
	  p_i(x) 		  & (j = i + 1) \\ 
	  u_ia_iz_i^{a_i - 1}  & (j = i) \\ 
	  v_ib_iz_{i - 1}^{b_i - 1} & (j = i + 1) \\ 
	  0 & ({\rm otherwise})
  	 \end{cases}
\]
For a maximal ideal $\mathfrak{m}$ of $B_n$, we denote by $J(\mathfrak{m})$ the Jacobian matrix at $\mathfrak{m}$, that is,
\begin{center}
	  $\displaystyle J(\mathfrak{m}) = {\Bigg (}\frac{\partial f_i}{\partial x} (\mathfrak{m}), \frac{\partial f_i}{\partial z_j} (\mathfrak{m}){\Bigg )_{1 \leq i \leq n,\: 0 \leq j \leq n + 1}}$
	\end{center} 
where for $g \in B_n$, $g(\mathfrak{m})$ means the image of $g$ in $B_n/\mathfrak{m}$. 

Take a common prime divisor $q(x) \in k[x]$ of $p_1(x),\ldots, p_n(x)$, which is possible since $p_i(x)\not\in k$ and each $p_i(x)$ generates the same radical ideal in $k[x]$. Let $\mathfrak{m}$ be the maximal ideal of $B_n$ generated by $q(x), z_0, \ldots, z_{n + 1}$. 
Since $a_i,b_i\ge 2$ for each $i$, we see that ${\rm rank} (J(\mathfrak{m})) = 0$, hence we have: 
	\begin{center}
	  $\dim_k (\mathfrak{m} / {\mathfrak{m}}^2) = (n + 3) - {\rm rank} (J(\mathfrak{m})) = n + 3$
	\end{center}
Therefore, the dimension of the tangent space at $\mathfrak{m}$ is $n + 3$, which implies $d \geq n + 3$. 
\end{proof}

\begin{example} {\rm In \cite{Daigle.Freudenburg.01a}, the authors give the following rings. Let $k$ be a field of characteristic zero, and let $p,q$ be prime integers with $p^2<q$. Given $n\ge 0$, define
\[
\Omega_n=k[X,Z_0,\hdots ,Z_{n+1}]/\left( XZ_{i+1}+Z_i^p+Z_{i-1}^q\right)_{1\le i\le n}
\]
where $X,Z_0,\hdots ,Z_{n+1}$ are indeterminates over $k$. Then for each $n\ge 0$, there exists a locally nilpotent derivation of $k^{[4]}$ with kernel isomorphic to  $\Omega_n$. Moreover, {\it Proposition~\ref{embed-dim}} shows that $\Omega_n$ is minimally generated by $n+3$ elements over $k$. }
\end{example}

\begin{remark} {\rm One motivation to consider UFDs of the type given in this section comes from the study of locally nilpotent derivations of polynomial rings $\C^{[n]}$ for the field $\C$ of complex numbers. The kernel $A$ of such a derivation is a UFD of transcendence degree $n-1$ over $\C$.
It is known that $A$ is quasi-affine \cite{Winkelmann.03} and ${\rm frac}(A)$ is ruled \cite{Deveney.Finston.94a}; that $A\cong \C^{[n-1]}$ if $1\le n\le 3$ \cite{Miyanishi.85}; and that 
$A$ is generally non-noetherian if $n\ge 5$ \cite{Daigle.Freudenburg.99, Freudenburg.Kuroda.17}. 
For $n=4$, it is further known that $A$ is rational \cite{Deveney.Finston.94a} and that there can be no {\it a priori} bound on the number of generators needed
\cite{Daigle.Freudenburg.01a}. But the question whether $A$ must be affine when $n=4$, or even noetherian, is open. We are thus led to study quasi-affine rational UFDs of transcendence degree 3 over $\C$. }
\end{remark}


\subsection{A Graded Rational Non-Noetherian UFD of Dimension 3 }
Let $k$ be a field and let $\Lambda =k[X,Z_0,Z_1,Z_2,\hdots ]$ be the polynomial ring in a countably infinite number of variables $X,Z_i$, $i\ge 0$. Define 
\[
\Omega = k[X,Z_0,Z_1,Z_2,\hdots ]/(X^{2^i}Z_{i+1}+Z_i+Z_{i-1}^2)_{i\ge 1}=k[x,z_0,z_1,z_2,\hdots ]
\]
where $x,z_i$ denote the images of $X,Z_i$ under the standard surjection of $\Lambda$ onto $\Omega$. 
For each $n\ge 1$, define:
\[
f_n=X^{2^n}Z_{n+1}+Z_n+Z_{n-1}^2\in \Lambda_n:=k[X,Z_0,\hdots ,Z_{n+1}]\cong k^{[n+3]}
\]
Since the ring displayed in (\ref{kernel}) is a UFD, hence a domain, 
we see that the ideal $I_n:=(f_1,\hdots ,f_n)$ is a prime ideal of $\Lambda_n$. 
Let $I\subset\Lambda$ be the ideal $I=(f_1,f_2,f_3,\hdots )$. Then $I=\bigcup_{n\ge 1} I_n$, which implies that $I$ is a prime ideal and $\Omega$ is an integral domain. 
In addition:
\[
\Omega [x^{-1}]=k[x,x^{-1},z_0,z_1]\cong k[x,x^{-1}]^{[2]} \implies {\rm frac}(\Omega )\cong k^{(3)}
\]
Define ideals in $\Omega$ by:
\[
\mathfrak{p} = (z_0,z_1,z_2,\hdots ) \quad {\rm and}\quad \mathfrak{m}=(x,z_0)
\]
\begin{lemma} $\mathfrak{m}$ is a maximal ideal of $\Omega$, and $x\,\Omega$ and $\mathfrak{p}$ are prime ideals properly contained in $\mathfrak{m}$. In addition:
\begin{itemize}
\item [{\bf (a)}] $z_i+z_0^{2^i}\in x\,\Omega$ for each $i\ge 1$, and
\item [{\bf (b)}] $z_i\not\in x\,\Omega$ for each $i\ge 0$.
\end{itemize}
\end{lemma}

\begin{proof} Define the subring $R\subset \Lambda$ and maximal ideal $J\subset R$ by:
\[
R=k[X,Z_i+Z_{i-1}^2]_{i\ge 1} \quad {\rm and}\quad J=(X,Z_i+Z_{i-1}^2)_{i\ge 1}
\]
Then $\Lambda = R[Z_0]\cong R^{[1]}$ and $J\Lambda = X\Lambda +I$. It follows that
\[
\Omega/x\,\Omega \cong \Lambda/J\Lambda = (R/J)[\bar{Z}_0] = k[\bar{Z}_0]\cong k^{[1]} \implies \Omega/x\,\Omega = k[\bar{z}_0]\cong k^{[1]}
\]
where $\bar{Z}_0$ is the image of $Z_0$ in $\Lambda/J\Lambda$ and $\bar{z}_0$ is the image of $z_0$ in $\Omega/x\,\Omega$. Therefore, $x\,\Omega$ is a prime ideal and $z_0\not\in x\,\Omega$.
For each $i\ge 1$, we have $z_i+z_{i-1}^2\in x\,\Omega$. Parts (a) and (b) now follow by induction.
Since $f_i\in  (Z_0,Z_1,Z_2,\hdots )$ for each $i\ge 1$, we see hat 
\[
\Omega/\mathfrak{p}\cong k[x]\cong k^{[1]}
\]
and $\mathfrak{p}$ is prime. Since $\Omega/\mathfrak{m}\cong k$, $\mathfrak{m}$ is maximal. Part (a) shows $\mathfrak{p}\subset\mathfrak{m}$. Since neither $x\,\Omega$ nor $\mathfrak{p}$ is maximal, their containment in $\mathfrak{m}$ is proper.
\end{proof}

Define a $\Z$-grading $\Omega =\bigoplus_{d\in\Z}\Omega_d$ by letting $x$ and each $z_i$ be homogeneous, $\deg x=-1$ and $\deg z_i=2^i$. 
A {\bf monomial} $\mu\in\Omega$ is of the form $\mu =x^r\prod_{i\in\N}z_i^{e_i}$ for $r,e_i\in\N$. Note that each monomial is homogeneous, and that the set of all monomials forms a multiplicative monoid. Define:
\[
\textstyle |\mu |=\min \Big\{ \sum_i\epsilon_i\, \Big |\, \mu=x^s\prod_iz_i^{\epsilon_i}\, ,\, s,\epsilon_i\in\N \Big\}
\]
Given $d\in\N$, let $d=\sum_{i\ge 0}d_i2^i$ be its binary expansion, where $d_i\in\{ 0,1\}$ for each $i$ and $d_i=0$ for all but finitely many $i$. 
Define the function: 
\[
\sigma :\N\to\Omega \,\, ,\,\, d\mapsto F_d=\prod_{i\ge 0}z_i^{d_i}
\]
Since $F_d\in\Omega_d$ for $d\in\N$, it follows that $\sigma$ is injective. 
Note that, since $z_i\not\in x\,\Omega$ for all $i\ge 0$, $F_d\not\in x\,\Omega$ for all $d\in\N$.

\begin{lemma}\label{basis} Given $d\in\Z$, the set $\mathcal{B}_d=\{ x^mF_n\, |\, m,n\in\N\, ,\, n-m=d\}$ is a $k$-basis of $\Omega_d$.
\end{lemma}

\begin{proof} Fix $d\in\Z$. 

Suppose that $c_1x^{m_1}F_{n_1}+\cdots +c_sx^{m_s}F_{n_s}=0$ for $c_i\in k^*$, $s\ge 2$, and distinct $(m_i,n_i)\in\N^2$ with $n_i-m_i=d$. 
Note that $m_1,\hdots ,m_s$ are distinct; we may assume $m_1=\min_im_i$.
Then $F_{n_1}\in x\,\Omega$ in contradiction to the above stated property.
Therefore, elements of $\mathcal{B}_d$ are $k$-linearly independent.

Let $V\subset\Omega_d$ be the subspace spanned by $\mathcal{B}_d$. We proceed by induction on $|\mu |$ to show that every monomial $\mu\in \Omega_d$ is in $V$. From this it follows that $V=\Omega_d$.

Let $\mu\in\Omega_d$ be a monomial where $\mu =x^r\prod_iz_i^{e_i}$ and $|\mu |=\sum_ie_i$. 

If $|\mu |=0$, then $\mu =x^rF_0\in V$. 

Assume that $\lambda\in V$ whenever $\lambda\in\Omega_d$ is a monomial with $0\le |\lambda |< |\mu |$. 
If $\max_ie_i\le 1$, then $\mu =x^rF_{d+r}\in V$. Otherwise, let $e_m=\max_ie_i\ge 2$ and write $e_m=2a+b$ for integers $a\ge 1$ and $b\in\{ 0,1\}$. 
Let $P=\mu/z_m^{e_m}$. Then:
\begin{eqnarray*}
\mu = Pz_m^{e_m} &=& Pz_m^b(z_m^2)^a \\
&=& (-1)^aPz_m^b(x^{2^{m+1}}z_{m+2}+z_{m+1})^a \\
&=& (-1)^aPz_m^b\sum_{j=0}^a{a\choose j}x^{j2^{m+1}}z_{m+2}^jz_{m+1}^{a-j} \\
&=& (-1)^a\sum_{j=0}^a{a\choose j}x^{j2^{m+1}}Pz_m^bz_{m+1}^{a-j}z_{m+2}^j 
\end{eqnarray*}
For each $j$, we see that $M_j:=x^{j2^{m+1}}Pz_m^bz_{m+1}^{a-j}z_{m+2}^j$ is a monomial in $\Omega_d$ and that:
\[
\textstyle |M_j|\le (\sum_{i\ne m}e_i) + b+(a-j)+j < \sum_ie_i = |\mu |
\]
By the inductive hypothesis, $M_j\in V$. It follows that $\mu\in V$.
\end{proof}

\begin{lemma}\label{main}
$\bigcap_{m\ge 0}x^m\Omega=(0)$
\end{lemma}

\begin{proof}
Since $J=\bigcap_{m\ge 0}x^m\Omega$ is a graded ideal, it suffices to show that $J \cap \Omega_d = 0$ for every $d \in \Z$.
Arguing by contradiction, assume that $0 \neq f \in J \cap \Omega_d$ for some $d \in \Z$. Then, by {\it Lemma\,\ref{basis}},
$f = \sum_{i\ge\max(0,-d)} a_i x^i F_{d+i}$ where $a_i \in k$ for all $i$. Let $m = \min\setspec{ i }{ a_i \neq 0 }$.
Then $\sum_{i>m}a_i x^i F_{d+i} \in x^{m+1}\Omega$ and (since $f \in J$) $f \in x^{m+1}\Omega$, so
$x^mF_{d+m} \in x^{m+1}\Omega$, so $F_{d+m} \in x\Omega$, a contradiction.
\end{proof}

\begin{theorem}\label{UFD-thm} $\Omega$ is a UFD.
\end{theorem}

\begin{proof}
Let $\qgoth = x\Omega$.
By {\it Lemma\,\ref{main}}, $\Omega_{\qgoth}$ is a discrete valuation ring. Since $\Omega = \Omega[x^{-1}]\cap\Omega_{\qgoth}$,
 {\it Prop.\,2.4} implies that $\Omega$ is a Krull domain. Since $\Omega [x^{-1}]$ is a UFD and $x$ is prime, it follows by Nagata's Criterion that $\Omega$ is a UFD.
\end{proof}

\begin{corollary} $z_0\Omega$ is a prime ideal properly contained in $\mathfrak{p}$. 
\end{corollary}

\begin{proof} Since $\Omega [x^{-1}]=k[x,x^{-1},z_0,z_1]\cong k[x,x^{-1}]^{[2]}$, $z_0$ is irreducible in $\Omega [x^{-1}]$. 
Since $z_0\not\in x\,\Omega$, $z_0$ is irreducible in $\Omega$. By {\it Theorem\,\ref{UFD-thm}}, $\Omega$ is a UFD. Therefore, $z_0$ is prime in $\Omega$. 

The inclusion $z_0\Omega\subset\mathfrak{p}$ is clear. 
Suppose that $z_1\in z_0\Omega$. Then 
\[
z_1\in z_0\Omega[x^{-1}] \,\, {\rm where}\,\, \Omega[x^{-1}]=k[x,x^{-1},z_0,z_1]\cong k[x,x^{-1}]^{[2]}
\]
which is not possible. Therefore, $z_0\Omega\subsetneq\mathfrak{p}$. 
\end{proof}

\begin{theorem} $\dim\Omega = 3$ and ${\rm ht}(\mathfrak{m})=3$. Consequently, $\Omega$ is not noetherian. 
\end{theorem}

\begin{proof} 
Since ${\rm frac}(\Omega )=k^{(3)}$ we see that ${\rm tr.deg}_k(\Omega )=3$. 
Since $\dim\Omega\le {\rm tr.deg}_k(\Omega )$ (see, for example, \cite{Kemper.11}, Theorem 5.5), 
we have $\dim\Omega\le 3$.
Since
\[
(0)\subsetneq z_0\Omega\subsetneq \mathfrak{p}\subsetneq\mathfrak{m}
\]
is a chain of primes in $\Omega$, it follows that ${\rm ht}(\mathfrak{m})=3$ and $\dim\Omega = 3$. Since $\mathfrak{m}$ is generated by two elements, 
it follows from the Krull dimension theorem (\cite{Eisenbud.95}, Theorem 10.2) that $\Omega$ cannot be noetherian. 
\end{proof}


\section{Application II: Affine UFDs Defined by Trinomial Relations}\label{app-two}

Let $k$ be a field. Assume that the following data are given.
\begin{itemize}
\item [(D.1)] An integer $n\ge 2$ and partition $n=n_0+n_1+\cdots +n_r$ where $r\ge 2$ and $n_i\ge 1$ for $0\le i\le r$. This induces a partition of variables $t_{ij}$ in the polynomial ring $k^{[n]}$:
\[
k^{[n]}=k[T_0,T_1,\hdots ,T_r] \,\,  ,\,\,  {\rm where}\,\, T_i=\{ t_{i1},\cdots ,t_{in_i}\}  \,\, ,\,\, 0\le i\le r
\]
\item [(D.2)] A sequence $\beta_0,\hdots ,\beta_r$ where $\beta_i=(\beta_{i1},\hdots ,\beta_{in_i})\in\Z_+^{n_i}$, $0\le i\le r$ satisfy:
\begin{quote}
\medskip
If $d_i=\gcd (\beta_{i1},\hdots ,\beta_{in_i})$, $0\le i\le r$, then $d_0,\hdots ,d_r$ are pairwise relatively prime. 
\end{quote}
\medskip
\noindent This induces a sequence $T^{\beta_0}_0,\hdots ,T^{\beta_r}_r\in k[T_0,\hdots ,T_r]$ of monomials: 
\[
T_i^{\beta_i}= t_{i1}^{\beta_{i1}}\cdots t_{in_i}^{\beta_{in_i}} \,\, ,\,\, 0\le i\le r
\]
\item [(D.3)] A sequence of distinct elements $\lambda_2,\hdots ,\lambda_r\in k^*$.
\end{itemize}

\begin{theorem}\label{trinomial} Given the data and notation in {\rm (D.1), (D.2)} and {\rm (D.3)} above, the ring
\[
B:=k[T_0,\hdots ,T_r]/(T_0^{\beta_0}+\lambda_iT_1^{\beta_1}+T_i^{\beta_i})_{2\le i\le r}
\]
is an affine rational UFD of dimension $n-r+1$ over $k$, and the image of $T_0^{\beta_0}+\mu T_1^{\beta_1}$ in $B$ is prime for each $\mu\in k^*\setminus\{ \lambda_2,\hdots ,\lambda_r\}$.
\end{theorem}

\begin{proof} In a slight abuse of notation, we use $T_i$ and $t_{ij}$ to denote their images in $B$. 
We proceed by induction on $r$. Define subrings $B_1\subset B_2\subset\cdots\subset B_r=B$ by:
\[
B_m=k[T_0,\hdots ,T_m]/(T_0^{\beta_0}+\lambda_iT_1^{\beta_1}+T_i^{\beta_i})_{2\le i\le m}
\]
Note that $B_1=k[T_0,T_1]\cong k^{[n_0+n_1]}$, which is a rational UFD. Moreover, it follows easily from  {\it Theorem\,\ref{Fifth-Criterion}(a)} that $T_0^{\beta_0}+\mu T_1^{\beta_1}$ is prime in $B_1$ for each $\mu\in k^*$. This gives a basis for induction. 

Assume that, for some $m\ge 2$, $B_1,\hdots ,B_{m-1}$ are rational UFDs over $k$ and that $T_0^{\beta_0}+\mu T_1^{\beta_1}$ is prime in $B_{m-1}$ for every 
$\mu\in k^*\setminus\{ \lambda_2,\hdots ,\lambda_{m-1}\}$.
Let $F=T_0^{\beta_0}+\lambda_mT_1^{\beta_1}$.  
Given $i$ with $0\le i\le m-1$, suppose that $d_i=d_{i1}\beta_{i1}+\cdots +d_{in_i}\beta_{in_i}$ for $d_{ij}\in\Z$. 
Put an $\N$-grading on $B_{m-1}$ by letting $t_{ij}$ be homogeneous with $\deg(t_{ij}) = d_{ij} d_0 \cdots\hat{d_i}\cdots d_{m-1}$.
Then for each monomial $T_i^{\beta_i}$, $0\le i\le m-1$, we have:
\[
\deg T_i^{\beta_i} = \deg (t_{i1}^{\beta_{i1}}\cdots t_{in_i}^{\beta_{in_i}})=(d_{i1}\beta_{i1}+\cdots +d_{in_i}\beta_{in_i})(d_0\cdots\hat{d_i}\cdots d_{m-1})=d_0\cdots d_{m-1}
\]
Therefore, $F$ is homogeneous and $\deg F=d_0\cdots d_{m-1}$. Since $B_m=B_{m-1}[T_m]/(T_m^{\beta_m}+F)$ and 
$\gcd (\beta_{m1},\hdots ,\beta_{mn_m},\deg F)=1$, {\it Theorem\,\ref{Fifth-Criterion}(b)} implies that $B_m$ is a UFD, and that ${\rm frac}(B_m)\cong {\rm frac}(B_{m-1})^{(n_m-1)}$. Since $B_{m-1}$ is rational over $k$, it follows that $B_m$ is rational over $k$. 

Given $\mu\in k^*\setminus\{ \lambda_2,\hdots ,\lambda_m\}$, let $G=T_0^{\beta_0}+\mu T_1^{\beta_1}$ and 
$\bar{B}_{m-1}=B_{m-1}/GB_{m-1}$. Since $G$ is a homogeneous prime of $B_{m-1}$, it follows that $\bar{B}_{m-1}$ is a $\Z$-graded integral domain. We have
\begin{eqnarray*}
B_m/GB_m &=& B_{m-1}[T_m]/(T_0^{\beta_0}+\mu T_1^{\beta_1} , T_m^{\beta_m} + T_0^{\beta_0} + \lambda_mT_1^{\beta_1}) \\
&=& \bar{B}_{m-1}[T_m]/(T_m^{\beta_m}+(\lambda_m-\mu)T_1^{\beta_1})
\end{eqnarray*}
and {\it Theorem\,\ref{Fifth-Criterion}(a)} implies that this ring is an integral domain. So $G$ is prime in $B_m$.

By induction, it follows that $B_r=B$ is a rational UFD over $k$ and $T_0^{\beta_0}+\mu T_1^{\beta_1}$ is prime in $B$ for each $\mu\in k^*\setminus\{ \lambda_2,\hdots ,\lambda_r\}$.
\end{proof}

\begin{example} {\rm In \cite{Mori.77}, Mori classified affine UFDs of dimension two over an algebraically closed field $k$ which admit a nontrivial $\N$-grading. Each such ring is of the form
\[
k[x,y,z_1,\hdots ,z_N]/(x^a+\mu_iy^b+z_i^{c_i})_{1\le i\le N}
\]
where $N\ge 0$; $a,b,c_1,\hdots ,c_N\ge 2$ are pairwise relatively prime; and $1=\mu_1,\hdots ,\mu_N\in k^*$ are distinct. These rings conform to the data (D.1) $n=N+2$, $r=N+1$, where $T_0=x$, $T_1=y$ and $T_i=z_{i-1}$, $2\le i\le r$; (D.2) $\beta_0=a$, $\beta_1=b$ and $\beta_i=c_{i-1}$, $2\le i\le r$; (D.3) $\lambda_i=\mu_{i-1}$, $2\le i\le r$. }
\end{example}

\begin{remark} {\rm A UFD $B$ of the type presented in {\it Theorem\,\ref{trinomial}} admits a non-degenerate\footnote{Let $(G,+)$ be an abelian group and
$R = \bigoplus_{i \in G} R_i$ a $G$-graded ring. We say that the grading is {\bf non-degenerate} if $\setspec{i \in G}{ R_i \neq 0}$ is a generating set of $G$.} grading by $\Z^{n-r}$. When $k$ is algebraically closed, this means that the variety $X={\rm Spec}(B)$ admits a torus action of complexity one. Surfaces and threefolds of this type were classified by Mori \cite{Mori.77} and Ishida \cite{Ishida.77}, respectively. 
More recently, Hausen, Herppich and S\"uss \cite{Hausen.Herppich.Suss.11} classified all such varieties in terms of Cox rings, under the additional assumption that the characteristic of $k$ is zero.  
Their description matches that given in {\it Theorem\,\ref{trinomial}} in that, when $k$ is algebraically closed of characteristic zero, every $k$-affine rational UFD of dimension $d$ that admits a non-degenerate $\Z^{d-1}$-grading is one of the rings of {\it Theorem\,\ref{trinomial}}. However, our}
description uses simpler data, aligning with Mori's description for surfaces. Moreover, {\it Theorem\,\ref{trinomial}} is valid for any field $k$.  
\end{remark}


\section{A Counterexample}\label{appendix}

The Bourbaki volume \cite{Bourbaki.72} includes the following exercise (p.549, Exercise 15(b), VII, \S\,1). 
\begin{quote} ``Let $A$ be a Krull domain and $a,b$ two [nonzero]\footnote{The hypothesis that $a,b\ne 0$ appears in part (a) of the exercise.} elements of $A$ such that $Aa$ and $Aa+Ab$ are prime and distinct. 
Show that $A[X]/(aX+b)$ is a Krull domain and that [the divisor class group] $C(A[X]/(aX+b))$ is isomorphic to $C(A)$.'' 
\end{quote}
In this section, we construct a counterexample to this assertion. 

Let $k$ be a field of characteristic zero and let $A=k[x,y]\cong k^{[2]}$.
Define a sequence of integers $s(n)$ by $s(1)=2,s(2)=3$ and $s(n)=n\prod_{1\le i\le n-2}s(i)$ for $n\ge 3$. 
Let $A[Z_0,Z_1,Z_2\hdots ]$ be the polynomial ring in a countably infinite number of variables $Z_i$ over $A$, and define:
\[
B=A[Z_0,Z_1,Z_2,\hdots ]/(xZ_{i+1}+y^{s(i+1)-1}Z_i^{s(i+1)}-Z_{i-1})_{i\ge 1} 
\]
\begin{theorem} The following properties hold.
\begin{itemize}
\item [{\bf (a)}] $B$ is a non-noetherian UFD of transcendence degree 4 over $k$, the surjection of $A[Z_0,Z_1,\hdots ]$ onto $B$ is injective on $A$, $x,y\in B$ are non-associated primes of $B$, and $xB+yB$ is a maximal ideal. 
\item [{\bf (b)}] Let $B[T]\cong B^{[1]}$. The ring $B'=B[T]/(xT-y)$ is an integral domain which does not satisfy the ascending chain condition on principal ideals.
\end{itemize}
\end{theorem}
The proof consists of the following series of lemmas.

\begin{lemma}  \label {bvuyr53wnd12ghjwerlio8deijf}
Let $k$ be a field, $R$ a $k$-algebra and $R_1 \subset R_2 \subset R_3 \subset \cdots$ finitely generated subalgebras of $R$ such that $R = \bigcup_{n\ge1} R_n$.
Consider $J_1 \subset J_2 \subset J_3 \subset \cdots$ where, for each $n \ge 1$, $J_n$ is a prime ideal of $R_n$.
Then $J = \bigcup_{n\ge1} J_n$ is a prime ideal of $R$.
Moreover, consider the following conditions:
\begin{itemize}

\item[\rm(i)] there exists a positive integer $d$ such that ${\rm tr.deg}_k(R_n/J_n)=d$ for all $n\ge 1$;

\item[\rm(ii)] for each $n\ge2$, $R_{n}/J_{n}$ is an algebraic extension of the image of $R_{n-1} \to R_n/J_n$;

\item[\rm(ii$'$)] for each $n\ge2$, $R_{n}/J_{n}$ is a birational extension of the image of $R_{n-1} \to R_n/J_n$.

\end{itemize}
Then the following hold.
\begin{enumerata}

\item If {\rm(i)} and {\rm(ii)} hold then $J \cap R_n = J_n$ and $\trdeg_k(R/J) = \trdeg_k(R_n/J_n)$ for all $n\ge1$.

\item If {\rm(i)} and {\rm(ii$'$)} hold then $J \cap R_n = J_n$ and $\Frac(R/J) = \Frac(R_n/J_n)$ for all $n\ge1$.

\end{enumerata}
\end{lemma}

\begin{proof}
It is clear that $J$ is a prime ideal of $R$. Assume that (i) and (ii) hold and let us prove:
\begin{equation} \label {9823ubdco8ew39wej}
\text{for all $m,n$ such that $1 \le m \le n$, we have $J_n \cap R_m = J_m$.}
\end{equation}
It suffices to prove the case where $m = n-1$.
So fix $n\ge2$ and let $\phi : R_{n-1} \to R_n/J_n$ be the composition $R_{n-1} \hookrightarrow R_n \to R_n/J_n$.
Assumptions (i) and (ii) give the first two equalities in:
\begin{equation*} 
\trdeg_k\big(R_{n-1}/J_{n-1}\big)
= \trdeg_k\big(R_n/J_n\big) = \trdeg_k\big( \phi(R_{n-1}) \big)
= \trdeg_k\big( R_{n-1}/(J_n\cap R_{n-1}) \big) .
\end{equation*}
We have $J_{n-1} \subset J_n \cap R_{n-1}$, where both $J_{n-1}$ and $J_n \cap R_{n-1}$ are prime ideals of $R_{n-1}$;
so the equality
$\trdeg_k\big(R_{n-1}/J_{n-1}\big) = \trdeg_k\big( R_{n-1}/(J_n\cap R_{n-1}) \big)$
implies that $J_{n-1} = J_n \cap R_{n-1}$, as desired. This proves \eqref{9823ubdco8ew39wej}.
It follows that
\begin{equation} \label {980UyyYeE92873teidbxcowa9es}
\text{$J \cap R_n = J_n$ for all $n\ge1$.}
\end{equation}
Indeed, consider $f \in J \cap R_n$. Then there exists $N\ge n$ such that $f \in J_N$,
so $f \in J_N \cap R_n = J_n$ by  \eqref{9823ubdco8ew39wej}, so \eqref{980UyyYeE92873teidbxcowa9es} is proved.

Let $S = R/J$, let $\pi : R \to S$ be the canonical surjection and let $S_n = \pi(R_n)$ for each $n\ge1$.
Then \eqref{980UyyYeE92873teidbxcowa9es} implies that $S_n \isom R_n/J_n$ for all $n\ge1$.
Moreover, assumption (ii) implies that $S_n$ is an algebraic extension of $S_{n-1}$ for every $n\ge2$;
since $S = \bigcup_{n\ge1} S_n$, it follows that $S$ is an algebraic extension of $S_n$  (so $\trdeg_k(R/J) = \trdeg_k(R_n/J_n)$) for all $n\ge1$.

If (i) and (ii$'$) hold then all of the above is true and $S_n$ is a birational extension of $S_{n-1}$ for every $n\ge2$;
so $S$ is a birational extension of $S_n$ (so $\Frac(R/J) = \Frac(R_n/J_n)$) for all $n\ge1$.
\end{proof}

\begin{lemma}  \label {0d2o938ue092uf0wi}
The canonical map $A \to B$ is injective, $B$ is a domain and $\Frac B = (\Frac A)^{(2)} = k^{(4)}$.
\end{lemma}

\begin{proof}
We use {\it Lemma \ref{bvuyr53wnd12ghjwerlio8deijf}} with $R = A[Z_0, Z_1,\dots]$ and, for each $n\ge1$,
$$
\text{$R_n = A[Z_0,\dots,Z_{n+1}]$ and $J_n = (xZ_{i+1}+y^{s(i+1)-1}Z_i^{s(i+1)}-Z_{i-1})_{i = 1}^n$.} 
$$
Let $J = \bigcup_{n\ge1} J_n = (xZ_{i+1}+y^{s(i+1)-1}Z_i^{s(i+1)}-Z_{i-1})_{i\ge 1}$, then $R/J = B$.
Since $J \subset (Z_0,Z_1,\dots)$ and  $(Z_0,Z_1,\dots) \cap A = (0)$, we have $J\cap A = (0)$, so $A\to B$ is injective. 
By the same argument, $A \to R_n/J_n$ is injective for each $n\ge1$.

Fix $n\ge1$.
For each $i \in \{1,\dots, n\}$, let $\phi_i$ be the $A$-automorphism of $R_n$ defined by $\phi_i(Z_{i-1}) = Z_{i-1} + xZ_{i+1}+y^{s(i+1)-1}Z_i^{s(i+1)}$
and $\phi_i(Z_j)=Z_j$ for all $j \in \{0,\dots,n+1\} \setminus \{i-1\}$.
Then $\phi_{n} \circ \cdots \circ \phi_1$ maps $J_n$ onto $(Z_0,\dots,Z_{n-1})$, so $R_n/J_n \isom A^{[2]}$.
So condition (i) of  {\it Lemma \ref{bvuyr53wnd12ghjwerlio8deijf}} is satisfied.
Write $R_n/J_n = A[z_0, \dots, z_{n+1}]$ where $z_i$ is the image of $Z_i$ by the canonical surjection $R_n\to R_n/J_n$.
Since  $A\to R_n/J_n$ is injective, we have $x \neq 0$ in $R_n/J_n$, so $xz_{n+1}+y^{s(n+1)-1}z_n^{s(n+1)}-z_{n-1}=0$ implies $z_{n+1} \in \Frac( A[z_{0},\dots,z_n])$,
showing that $R_n/J_n$ is a birational extension of the image of $R_{n-1} \to R_n/J_n$, i.e.,
condition (ii$'$) of  {\it Lemma \ref{bvuyr53wnd12ghjwerlio8deijf}} is satisfied.
By that lemma, $B$ is a domain and $\Frac B = \Frac( R_n/J_n ) = (\Frac A)^{(2)} = k^{(4)}$.
\end{proof}

\begin{lemma}  \label {pc9v987237edgw9s0AAA1}
$B/xB$ is a domain of transcendence degree $2$ over $k$.
\end{lemma}

\begin{proof}
We use {\it Lemma \ref{bvuyr53wnd12ghjwerlio8deijf}} with $R = A[Z_0, Z_1,\dots]$ and, for each $n\ge1$,
$$
\text{$R_n = A[Z_0,\dots,Z_{n+1}]$ and $J_n = (x) + (xZ_{i+1}+y^{s(i+1)-1}Z_i^{s(i+1)}-Z_{i-1})_{i = 1}^n$.} 
$$
Let $J = \bigcup_{i\ge1}J_n$, then $B/xB = R/J$. Fix $n\ge1$. Then  $J_n = (x) + (y^{s(i+1)-1}Z_i^{s(i+1)}-Z_{i-1})_{i = 1}^n$. 
If we write $V_{i-1} = y^{s(i+1)-1}Z_i^{s(i+1)}-Z_{i-1}$ for $i=1,\dots,n$, then $R_n=k[x,y,V_0,\dots,V_{n-1},Z_n,Z_{n+1}]$
and $J_n = (x,V_0,\dots,V_{n-1})$, so $R_n/J_n = k^{[3]}$ and $J_n$ is a prime ideal of $R_n$ of height $n+1$.
It follows from {\it Lemma \ref{bvuyr53wnd12ghjwerlio8deijf}} that $J$ is a prime ideal of $R$ and hence that $B/xB$ is a domain.
Let us prove:
\begin{equation}  \label {c093y5iwer2i3vvlf5a4wu}
\text{the composition $k[y,Z_0] \hookrightarrow R \to R/J = B/xB$ is injective.}
\end{equation}
To see this, it's enough to show that $k[y,Z_0] \cap J = (0)$;
so it's enough to show that $k[y,Z_0] \cap J_n = (0)$ for all $n\ge1$.
Let $n\ge1$.
If $y \in J_n$ then we see that $Z_0, \dots, Z_{n-1} \in J_n$ (and $x,y \in J_n$) so $\haut J_n > n+1$, a contradiction.
So $y \notin J_n$.
If $Z_0 \in J_n$ then (using that $J_n$ is prime and that $y \notin J_n$) we see that $Z_0, \dots, Z_{n} \in J_n$ (and $x \in J_n$), so $\haut J_n > n+1$, a contradiction.
So $Z_0 \notin J_n$.
Consider the $k$-homomorphism $\pi : R_n \to k[y,Z_0]$ defined by $\pi(x)=0 = \pi(Z_1) = \cdots \pi(Z_{n+1})$, $\pi(y)=y$ and $\pi(Z_0)=Z_0$.
If $f$ is a nonzero element of $k[y,Z_0] \cap J_n$, write $f = Z_0^m g$ where $m \in \N$, $g \in k[y,Z_0]$ and $Z_0 \nmid g$;
since $J_n$ is prime and $Z_0 \notin J_n$, we have $g \in J_n$;
write $g = x g_0 + \sum_{i=1}^n (y^{s(i+1)-1}Z_i^{s(i+1)}-Z_{i-1})g_i$ ($g_0, \dots, g_n \in R_n$)
and note that $g = \pi(g) = - Z_0 \pi(g_1)\in Z_0 k[y,Z_0]$, a contradiction.
This shows that $f$ does not exist, i.e., $k[y,Z_0] \cap J_n = (0)$.
So \eqref{c093y5iwer2i3vvlf5a4wu} is true.

Write $B/xB = k[y,z_0, z_1, \dots]$ where $z_i$ is the canonical image of $Z_i$ in $R/J=B/xB$.
Then $y^{s(i+1)-1}z_i^{s(i+1)}-z_{i-1}=0$ in $B/xB$ for all $i\ge1$.
Since $y \neq 0$ in $B/xB$, it follows that $z_i$ is algebraic over $k(y,z_0,\dots,z_{i-1})$ for all $i\ge1$.
So $\Frac(B/xB)$ is an algebraic extension of $k(y,z_0)$.  We have $k(y,z_0) = k^{(2)}$ by \eqref{c093y5iwer2i3vvlf5a4wu},
so $\trdeg_k(B/xB) = 2$.
\end{proof}

\begin{lemma}  \label {pc9v987237edgw9s0AAA2}
$B/yB$ is a domain and $\Frac(B/yB) = k^{(3)}$.
\end{lemma}

\begin{proof}
We use {\it Lemma \ref{bvuyr53wnd12ghjwerlio8deijf}} with $R = A[Z_0, Z_1,\dots]$ and, for each $n\ge1$,
$$
\text{$R_n = A[Z_0,\dots,Z_{n+1}]$ and $J_n = (y) + (xZ_{i+1}+y^{s(i+1)-1}Z_i^{s(i+1)}-Z_{i-1})_{i = 1}^n$.} 
$$
Let $J = \bigcup_{i\ge1}J_n$, then $B/yB = R/J$. Fix $n\ge1$. Then  $J_n = (y) +  (xZ_{i+1}-Z_{i-1})_{i = 1}^n$.
If we define $V_{i-1} = xZ_{i+1}-Z_{i-1}$ for $i=1,\dots,n$, then $R_n = k[x,y,V_0,\dots,V_{n-1},Z_n,Z_{n+1}]$
and $J_n=(y,V_0,\dots,V_{n-1})$, so $R_n/J_n = k^{[3]}$ and $J_n$ is a prime ideal of $R_n$ of height $n+1$.
By {\it Lemma \ref{bvuyr53wnd12ghjwerlio8deijf}}, $J$ is a prime ideal of $R$, so $B/yB$ is a domain.
Note that $x \neq 0$ in $R_n/J_n$ (otherwise we would have $x \in J_n$, so $x,y,Z_0,\dots,Z_{n-1} \in J_n$, so $\haut J_n > n+1$, a contradiction).
Write $R_n/J_n = k[x,z_0,\dots,z_{n+1}]$ where $z_i$ is the canonical image of $Z_i$ in $R_n/J_n$.
Since $xz_{n+1}-z_{n-1} = 0$ and $x \neq 0$ in $R_n/J_n$, we have $z_{n+1} \in k(x,z_0,\dots,z_n)$,
so condition (ii$'$) of {\it Lemma \ref{bvuyr53wnd12ghjwerlio8deijf}} is satisfied.
It follows that $\Frac(B/yB) = \Frac(R_n/J_n) = k^{(3)}$.
\end{proof}

\begin{lemma} \label {dA09fij9283gwidw201nbw6}
$B$ is a UFD, $x,y$ are prime elements of $B$ such that $x \nmid y$, and $\mathfrak{m}:=xB+yB$ is a maximal ideal of $B$. 
\end{lemma}

\begin{proof}
Define $L=k(y)$, $S = k[y]\setminus\{0\}$,  $A_L=S^{-1}A=L[x]$, and $B_L = S^{-1}B$. Then
$$
B_L= L[x][Z_0,Z_1,Z_2,\hdots ]/(xZ_{i+1}+y^{s(i+1)-1}Z_i^{s(i+1)}-Z_{i-1})_{i\ge 1} .
$$
According to \cite{David.73} the ring
\[
R:=L[x][Z_0,Z_1,Z_2,\hdots ]/(xZ_{i+1}+Z_i^{s(i+1)}-Z_{i-1})_{i\ge 1}
\]
is a UFD. Let $\phi$ be the $L[x]$-automorphism of $L[x][Z_0,Z_1,Z_2,\hdots ]$ defined by $\phi(Z_i) = yZ_i$ for all $i\ge0$.
Then $\phi$ maps the ideal $(xZ_{i+1}+Z_i^{s(i+1)}-Z_{i-1})_{i\ge 1}$ of $L[x][Z_0,Z_1,Z_2,\hdots]$ onto the ideal $(xZ_{i+1}+y^{s(i+1)-1}Z_i^{s(i+1)}-Z_{i-1})_{i\ge 1}$ 
of $L[x][Z_0,Z_1,Z_2,\hdots]$.  So $B_L\cong_LR$ and $B_L$ is a UFD.

Write $B = k[x,y,z_0,z_1,\dots]$ where $z_i$ is the canonical image of $Z_i$ in $B$.
Then
\[
xz_{i+1}+y^{s(i+1)-1}z_i^{s(i+1)}-z_{i-1} = 0
\]
in $B$ for all $i\ge1$;
since $x \neq 0$ in $B$ (because $A \to B$ is injective), it follows that $B_x = A_x[z_0,z_1]$.
Since $\Frac B = k^{(4)}$ by {\it Lemma \ref{0d2o938ue092uf0wi}}, $z_0,z_1$ must be algebraically independent over $A$. So
$$
\text{$B_x = A_x^{[2]}$ is a UFD.}
$$
Next, we claim that $B = B_x \cap B_L$.
Indeed, define $B_n=A[z_n,z_{n+1}]$ for all $n\ge0$.
Since 
\[
xz_{i+1}+y^{s(i+1)-1}z_i^{s(i+1)}-z_{i-1} = 0
\]
for all $i\ge1$, we have $z_n \in B_{n+1}$ and hence $B_n\subset B_{n+1}$ for all $n\ge0$.
This gives the filtration $B=\bigcup_{n\ge 0}B_n$.
Since $z_0,z_1$ are algebraically independent over $A$, and since $z_0,z_1 \in B_n=A[z_n,z_{n+1}]$ for all $n\ge0$,
$z_n,z_{n+1}$ are algebraically independent over $A$ and hence $B_n = A^{[2]}$ for all $n\ge 0$.
So 
\[
B_x =\bigcup_{n\ge 0} A_x[z_n,z_{n+1}] \quad {\rm and}\quad B_L =\bigcup_{n\ge 0} A_L[z_n,z_{n+1}]
\]
where $A_x[z_n,z_{n+1}] = A_x^{[2]}$ and $A_L[z_n,z_{n+1}] = A_L^{[2]}$ for all $n\ge 0$.
Given $f\in B_x\cap B_L$, choose $n\ge 0$ so that $f\in A_x[z_n,z_{n+1}] \cap A_L[z_n,z_{n+1}]$ and write:
\[
\textstyle f=\sum_{i,j}\alpha_{ij}z_n^iz_{n+1}^j = \sum_{i,j}\beta_{ij}z_n^iz_{n+1}^j \quad (\alpha_{ij}\in A_x\, ,\, \beta_{ij}\in A_L) .
\]
By uniqueness of coefficients in $\Frac(A)[z_n,z_{n+1}] \isom \Frac(A)^{[2]}$,
we see that $\alpha_{ij}=\beta_{ij}$ for each pair $i,j\ge 0$,
so these coefficients belong to $A_x\cap A_L=A$. So $f\in B_n\subset B$, and $B=B_x \cap B_L$ is proved.

Since $B$ is the intersection of two UFDs, it is a Krull domain.
Since $x$ is prime in $B$ and $B_x$ is a UFD, it follows by Nagata's Criterion that $B$ is a UFD.

Since (by {\it Lemma\,\ref{pc9v987237edgw9s0AAA1}} and {\it Lemma\,\ref{pc9v987237edgw9s0AAA2}}) $\trdeg_k(B/xB) = 2$ and $\trdeg_k(B/yB) = 3$,
the prime ideals $xB$ and $yB$ are nonzero and distinct, so $x,y$ are prime elements of $B$ and $x \nmid y$.
We have $B/\mathfrak{m} = k[x,y,Z_0,Z_1,\dots] / J$ where 
$$
J = (x,y)+ (xZ_{i+1}+y^{s(i+1)-1}Z_i^{s(i+1)}-Z_{i-1})_{i\ge 1} = (x,y,Z_0,Z_1,\dots),
$$
so $B/\mathfrak{m} = k$ and $\mathfrak{m}$ is a maximal ideal of $B$.
\end{proof}

We continue to use the notation $B = k[x,y,z_0,z_1,\dots]$ where $z_i$ is the canonical image of $Z_i$ in~$B$.

\begin{lemma}\label{non-ACCP}
We have $0 \neq z_0\in\bigcap_{n\ge 1}\mathfrak{m}^n$.
Moreover, $B$ is not noetherian.
\end{lemma} 

\begin{proof}
For each $i\ge1$, we have $xz_{i+1}+y^{s(i+1)-1}z_i^{s(i+1)}-z_{i-1}=0$ and hence $z_{i-1} \in\mathfrak{m} z_{i} + \mathfrak{m} z_{i+1}$.
So $z_0 \in  \sum_{j\ge1} \mathfrak{m} z_j$ and, for each $n\ge1$: 
\[
\sum_{j\ge1} \mathfrak{m}^n z_j \subset \sum_{j\ge1} \mathfrak{m}^n (\mathfrak{m} z_{j+1} + \mathfrak{m} z_{j+2}) \subset \sum_{j\ge1} \mathfrak{m}^{n+1} z_j
\]
It follows that $z_0 \in \sum_{j\ge1} \mathfrak{m}^n z_j$ for all $n\ge1$, so $z_0 \in \bigcap_{n\ge1}\mathfrak{m}^n$.
Note that $z_0\ne 0$, because ${\rm frac}(B)=k(x,y,z_0,z_1)\cong k^{(4)}$. 
Since $\mathfrak{m}$ is a proper ideal of the domain $B$, the fact that $\bigcap_{n\ge 1}\mathfrak{m}^n \neq (0)$ implies that
$B$ is non-noetherian, by the Krull Intersection Theorem (\cite{Eisenbud.95}, Corollary 5.4). 
\end{proof}

\begin{lemma}  \label {o9c8f8721w3w3q3wqaszyq}
Let $B'=B[T]/(xT-y)$, where $B[T]\cong B^{[1]}$. Then $B'\cong B[y/x]\subset {\rm frac}(B)$ and $B'$ does not satisfy the ascending chain condition on principal ideals.
In particular, $B'$ is not a UFD.
\end{lemma}

\begin{proof} By {\it Lemma\,\ref{dA09fij9283gwidw201nbw6}}, $B$ is a UFD and $x,y\in B$ are nonzero relatively prime elements. By Samuel's Criterion 
({\it Theorem\,\ref{Samuel-lemma}}), $B'$ is an integral domain isomorphic to $B[y/x]$. 
Since $y\in xB'$, {\it Lemma\,\ref{non-ACCP}} shows that:
\[
0 \neq z_0\in\bigcap_{n\ge 0} x^nB' .
\]
Using that $\mathfrak{m}$ is a proper ideal of $B$ and that $x,y$ are relatively prime in $B$, we find that $x$ is not a unit of $B'$.
Consequently, $B'$ does not satisfy the ascending chain condition on principal ideals, so $B'$ is not a UFD. 
\end{proof}

\medskip

\noindent {\bf Acknowledgments.} The work of the third author was supported by JSPS Overseas Challenge Program for Young Researchers (No. 201880243) and JSPS KAKENHI Grant Numbers JP18J10420 and JP20K22317.
The authors gratefully acknowledge the helpful advice of S. Bhatwadekar and Xiasong Sun in preparing this article. The third author wishes to express his gratitude to members of the Department of Mathematics at Western Michigan University, which he visited for periods in 2018 and 2019. 



\begin{thebibliography}{10}

\bibitem{Bourbaki.72}
N.~Bourbaki, \emph{Commutative {A}lgebra}, Elements of Math., Addison-Wesley,
  1972, Reading, MA.

\bibitem{Cohn.77}
P.~Cohn, \emph{Algebra, {V}ol. 2}, John Wiley and Sons, 1977, Chichester, New
  York.

\bibitem{Daigle.Freudenburg.99}
D.~Daigle and G.~Freudenburg, \emph{A counterexample to {H}ilbert's
  {F}ourteenth {P}roblem in dimension five}, J. Algebra \textbf{221} (1999),
  528--535.

\bibitem{Daigle.Freudenburg.01a}
\bysame, \emph{A note on triangular derivations of {$k[X_1,X_2,X_3,X_4]$}},
  Proc. Amer. Math. Soc. \textbf{129} (2001), 657--662.

\bibitem{David.73}
J.~David, \emph{A characteristic zero non-noetherian factorial ring of
  dimension three}, Trans. Amer. Math. Soc. \textbf{180} (1973), 315--325.

\bibitem{Deveney.Finston.94a}
J.~Deveney and D.~Finston, \emph{Fields of {${\mathbb G}_a$} invariants are
  ruled}, Canad.\ Math.\ Bull. \textbf{37} (1994), 37--41.

\bibitem{Eisenbud.95}
D.~Eisenbud, \emph{Commutative {A}lgebra with a {V}iew {T}oward {A}lgebraic
  {G}eometry}, Graduate {T}exts in {M}athematics, vol. 150, Springer-{V}erlag,
  1995.

\bibitem{Fossum.73}
R.~Fossum, \emph{The {D}ivisor {C}lass {G}roup of a {K}rull {D}omain},
  Ergebnisse der Mathematik und ihrer Grenzgebiete, vol.~74, Springer-Verlag,
  Berlin, Heidelberg, New York, 1973.

\bibitem{Freudenburg.17}
G.~Freudenburg, \emph{Algebraic {T}heory of {L}ocally {N}ilpotent
  {D}erivations}, second ed., Encyclopaedia of Mathematical Sciences, vol. 136,
  Springer-Verlag, Berlin, Heidelberg, New York, 2017.

\bibitem{Freudenburg.Kuroda.17}
G.~Freudenburg and S.~Kuroda, \emph{Cable algebras and rings of {${\mathbb
  G}_a$}-invariants}, Kyoto J. Math. \textbf{57} (2017), 325--363.

\bibitem{Hausen.Herppich.Suss.11}
J.~Hausen, E.~Herppich, and H.~S\"uss, \emph{Multigraded factorial rings and
  {F}ano varieties with torus actions}, Doc. Math. \textbf{16} (2011), 71--109.

\bibitem{Ishida.77}
M.~Ishida, \emph{Graded factorial rings of dimension 3 of a restricted type},
  J. Math. Kyoto Univ. \textbf{17} (1977), 441--456.

\bibitem{Kaplansky.74}
I.~Kaplansky, \emph{Commutative {R}ings ({R}evised {E}dition)}, University of
  Chicago Press (Chicago), 1974.

\bibitem{Kemper.11}
G.~Kemper, \emph{A {C}ourse in {C}ommutative {A}lgebra}, Graduate Texts in
  Mathematics, vol. 256, Springer Verlag, 2011.

\bibitem{Lang.93}
S.~Lang, \emph{Algebra ({T}hird {E}dition)}, Addison Wesley, 1993.

\bibitem{Matsumura.80}
H.~Matsumura, \emph{{C}ommutative {A}lgebra ({S}econd {E}dition)}, Math.
  Lecture Note Series, vol.~56, Benjamin/Cummings Publishing (Reading, Mass.),
  1980.

\bibitem{Matsumura.86}
\bysame, \emph{{C}ommutative {R}ing {T}heory}, Cambridge Studies in Adv. Math.,
  vol.~8, Cambridge University Press, 1986.

\bibitem{Miyanishi.85}
M.~Miyanishi, \emph{Normal affine subalgebras of a polynomial ring}, Algebraic
  and {T}opological {T}heories---to the memory of Dr.~Takehiko {M}iyata
  (Tokyo), Kinokuniya, 1985, pp.~37--51.

\bibitem{Mori.77}
S.~Mori, \emph{Graded factorial domains}, Japan J. Math. \textbf{3} (1977),
  224--238.

\bibitem{Nagata.57}
M.~Nagata, \emph{A remark on the unique factorization theorem}, J. Math. Soc.
  Japan \textbf{9} (1957), 143--145.

\bibitem{Samuel.64}
P.~Samuel, \emph{Lectures on unique factorization domains}, Tata Institute of
  Fundamental Research Lectures on Mathematics, vol.~30, Tata Institute of
  Fundamental Research, Bombay, 1964.

\bibitem{Winkelmann.03}
J.~Winkelmann, \emph{Invariant rings and quasiaffine quotients}, Math. Z.
  \textbf{244} (2003), 163--174.

\end{thebibliography}

\vspace{.2in}

\noindent \address{Department of Mathematics. and Statistics\\
University of Ottawa\\
Ottawa, Ontario K1N 6N5}\\
Canada\\
\email{ddaigle@uottawa.ca}
\bigskip

\noindent \address{Department of Mathematics\\
Western Michigan University\\
Kalamazoo, Michigan 49008}\\
USA\\
\email{gene.freudenburg@wmich.edu}
\bigskip

\noindent\address{National Institute of Technology (KOSEN)\\
Oyama College\\
771 Nakakuki, Oyama, Tochigi 323-0806\\
Japan\\
\email{t.nagamine14@oyama-ct.ac.jp}

\end{document}